\RequirePackage{fix-cm}
\documentclass[smallextended]{svjour3}       
\smartqed  

\smartqed 

\usepackage{amsmath}
\usepackage{amssymb}
\usepackage{amsfonts}
\usepackage{hyphenat}
\usepackage{hyperref}
\usepackage{times}
\usepackage{psfrag}
\usepackage{ifpdf}
\usepackage{array}
\usepackage{multirow}
\usepackage{setspace}
\usepackage{latexsym}
\usepackage{enumerate}
\usepackage{mathtools}
\usepackage{float}
\usepackage{scalerel}
\usepackage{cite}
\usepackage{graphicx}
\usepackage{caption}
\usepackage{epstopdf}
\usepackage[table]{xcolor}
\usepackage{subfig}
\usepackage{accents}
\usepackage{algorithm}
\usepackage{algpseudocode}
\usepackage{mathrsfs}
\usepackage{tabularx}
\usepackage{tikz}
\usepackage{pgfplots}
\usepackage{url}  
\usepackage{caption}
\usepackage{rotating}
\usepackage{adjustbox}

\usepackage[numbers,sort&compress]{natbib}

\definecolor{Fcolor}{rgb}{0,0.5,0.25}

\newcommand{\ignore}[1]{}

\newcommand{\tr}{\mathrm{tr}}

\newcommand{\T}{^{\!\top}}
\newcommand{\Tp}{^{\!\phantom{\top}}}

\newcommand{\la}{\langle}
\newcommand{\ra}{\rangle}

\newcommand{\asf}{\mathsf{a}}
\newcommand{\abf}{\boldsymbol{a}}
\newcommand{\Asf}{\boldsymbol{\mathsf{A}}}
\newcommand{\Abf}{\boldsymbol{A}}

\newcommand{\Bsf}{\boldsymbol{\mathsf{B}}}
\newcommand{\Bbf}{\boldsymbol{B}}
\newcommand{\Bcal}{\mathcal{B}}

\newcommand{\Bst}{\accentset{\ast}{\Bcal}}

\newcommand{\Ccal}{\mathcal{C}}

\newcommand{\Dbb}{\mathbb{D}}

\newcommand{\ebf}{\boldsymbol{e}}

\newcommand{\Ecal}{\mathcal{E}}

\newcommand{\Fcal}{\mathcal{F}}
\newcommand{\Frm}{\mathrm{F}}

\newcommand{\Fch}{\check{\mathcal{F}}}
\newcommand{\Fsf}{\boldsymbol{\mathsf{F}}}

\newcommand{\Ibf}{\boldsymbol{I}}

\newcommand{\Ical}{\mathcal{I}}

\newcommand{\Jbf}{\boldsymbol{J}}

\newcommand{\Kcal}{\mathcal{K}}

\newcommand{\msf}{\mathsf{m}}

\newcommand{\nsf}{\mathsf{n}}

\newcommand{\Ncal}{\mathcal{N}}

\newcommand{\Psf}{\boldsymbol{\mathsf{P}}}

\newcommand{\qbf}{\boldsymbol{q}}

\newcommand{\Qcal}{\mathcal{Q}}
\newcommand{\qch}{\check{q}}

\newcommand{\Rbb}{\mathbb{R}}

\newcommand{\Sbb}{\mathbb{S}}
\newcommand{\Scal}{\mathcal{S}}

\newcommand{\tsf}{\mathsf{t}}

\newcommand{\Tcal}{\mathcal{T}}

\newcommand{\usf}{\boldsymbol{\mathsf{u}}}

\newcommand{\Ucal}{\mathcal{U}}

\newcommand{\wsf}{\boldsymbol{\mathsf{w}}}

\newcommand{\xsf}{\boldsymbol{\mathsf{x}}}
\newcommand{\xbf}{\boldsymbol{x}}
\newcommand{\Xbf}{\boldsymbol{X}}

\newcommand{\xst}{\accentset{\ast}{\xbf}}
\newcommand{\Xst}{\accentset{\ast}{\Xbf}}
\newcommand{\Xsf}{\boldsymbol{\mathsf{X}}}

\newcommand{\Ybf}{\boldsymbol{Y}}
\newcommand{\Ych}{\check{\Ybf}}
\newcommand{\Yst}{\accentset{\ast}{\Ybf}}

\newcommand{\zsf}{\boldsymbol{\mathsf{z}}}

\newcommand{\xibf}{\boldsymbol{\xi}}

\newcommand{\nubf}{\boldsymbol{\nu}}
\newcommand{\taubf}{\boldsymbol{\tau}}

\newcommand{\taust}{\accentset{\ast}{\taubf}}

\newcommand{\Lambdabf}{\boldsymbol{\Lambda}}
\newcommand{\Lambdast}{\accentset{\ast}{\Lambdabf}}

\makeatletter
\DeclareRobustCommand{\cev}[1]{%
  \mathpalette\do@cev{#1}%
}
\newcommand{\do@cev}[2]{%
  \fix@cev{#1}{+}%
  \reflectbox{$\m@th#1\vec{\reflectbox{$\fix@cev{#1}{-}\m@th#1#2\fix@cev{#1}{+}$}}$}%
  \fix@cev{#1}{-}%
}
\newcommand{\fix@cev}[2]{%
  \ifx#1\displaystyle
    \mkern#2 1mu
  \else
    \ifx#1\textstyle
      \mkern#2 3mu
    \else
      \ifx#1\scriptstyle
        \mkern#2 2mu
      \else
        \mkern#2 2mu
      \fi
    \fi
  \fi
}

\begin{document}

\title{Parabolic Relaxation for Quadratically-constrained Quadratic Programming -- Part II: Theoretical and Computational Results
	\thanks{This work is in part supported by the NSF Award 1809454. Alper Atamt\"{u}rk is supported, in part, by grant FA9550-10-1-0168 from the Office of the Assistant Secretary of Defense for Research \& Engineering and the NSF Award 1807260.}
}


\author{Ramtin Madani \and Mersedeh Ashraphijuo \and Mohsen Kheirandishfard \and Alper Atamt\"{u}rk}

\institute{Ramtin Madani \and Mersedeh Ashraphijuo \at
	Department of Electrical Engineering, University of Texas Arlington \\
	\email{ramtin.madani@uta.edu, mersedeh.ashraphijuo@uta.edu}           
	\and
	Mohsen Kheirandishfard \at
	Cognitiv \\
	\email{mohsen.kheirandishfard@gmail.com}
	\and
	Alper Atamt\"{u}rk \at
	Department of Industrial Engineering \& Operations Research, University of California Berkeley \\
	\email{atamturk@berkeley.edu}   
}

\date{Received: date / Accepted: date}

\maketitle


\begin{abstract}
	In the first part of this work \cite{parabolic_part1}, we introduce a convex {\it parabolic relaxation} for quadratically-constrained quadratic programs, along with a sequential penalized
	parabolic relaxation algorithm to recover near-optimal feasible solutions.
   In this second part, we show that starting from a feasible solution or a near-feasible solution satisfying certain regularity conditions, the sequential penalized
	parabolic relaxation algorithm convergences to a point which satisfies Karush–Kuhn–Tucker optimality conditions. Next, we present
	numerical experiments on benchmark non-convex QCQP problems as well as large-scale instances of system identification problem demonstrating the efficiency of the proposed approach.
	
\end{abstract}

\section{Introduction}

As discussed in the first part of this work, \cite{parabolic_part1}, we study the problem of minimizing a quadratic objective function over a feasible set that is also defined by quadratic functions, i.e., quadratically-constrained quadratic programming (QCQP). A wide variety of convex relaxations and search methods have been proposed in the literature to tackle the class of QCQP optimization problems. The most notable among them is the semidefinite programming (SDP) relaxation.

Scalability is the main challenge to the real-world applications of SDP. Recently, there have been a significant attention towards addressing this issue \cite{majumdar2019recent}. One approach for improvement is developing methods to exploit problem-specific structures, e.g., sparsity and symmetry. For certain structures, there are techniques that allow expressing a large semidefinite constraint as a set of smaller semidefinite constraints, with increased computational efficiency \cite{agler1988positive,grone1984positive}. Such results are leveraged in optimization and control theory problems to improve efficiency \cite{fujisawa2009user,fukuda2001exploiting,kim2011exploiting,mason2014chordal}. By exploiting the underlying sparsity properties of the problem, and using decomposition, \cite{zheng2016chordal,zheng2017scalable} propose scalable algorithms to obtain structured feedback gains to stabilize large-scale systems. Also, the sparsity structure is exploited for polynomial optimization problems in \cite{lasserre2017bounded,weisser2018sparse}. Another common structure in applications of SDP is symmetry. To this end, the method of symmetry reduction proposed in \cite{de2010exploiting,gatermann2004symmetry,permenter2017reduction,vallentin2009symmetry}. 
Similar to symmetry reduction, facial reduction leverages a degeneracy structure ad proposed in \cite{borwein1981regularizing,drusvyatskiy2017many,kungurtsev2020two,pataki2013strong,permenter2017reduction}.

Another approach is developing methods for producing low-rank solutions. In some cases where there is indeed a low-rank optimal solution to the problem \cite{AG:rank-one, Barvinok95,lemon2016low,Pataki98}, and in other cases acceptable low-rank feasible points to the problem are sufficient \cite{d2004direct,kulis2007fast,Recht07}. One direction includes the Burer-Monteiro algorithm \cite{burer2003nonlinear}, quasi-Newton algorithms for unconstrained optimization problems \cite{liu1989limited}, Riemannian optimization \cite{boumal2016non,journee2010low}, and coordinate descent \cite{erdogdu2021convergence,sun2015decomposition}. Another direction includes Frank-Wolfe based methods \cite{frank1956algorithm} such as Hazan's algorithm \cite{hazan2008sparse} and the numerical method in \cite{yurtsever2017sketchy}. 

First-order methods scale to significantly larger problem sizes, while trading off the accuracy of resulting output, e.g., Alternating Direction Method of Multipliers (ADMM) \cite{o2016conic}, and the augmented Lagrangian-based methods \cite{yang2015sdpnal,zhao2010newton}.

And finally, another approach that is most relevant to our work is a set of conservative algorithms that lead to guaranteed feasible points which can be suboptimal, as opposed to the ones which produce solutions that may violate some constraints. They are special-purpose algorithms leveraging domain knowledge associated with the target application, e.g., in neural networks \cite{fazlyab2020safety,qin2018verification,raghunathan2018semidefinite}, or general-purpose ones that apply broadly across application domains, e.g., DSOS and SDSOS algorithms \cite{ahmadi2017dsos} and their adaptive improvements where a sequence of linear programs (LPs) or second order cone programs (SOCPs) are solved, e.g., by iterative change of basis \cite{ahmadi2017sum} or column generation \cite{ahmadi2017optimization}.

In the first part of this work, we introduce {\it parabolic relaxation}, an alternative convex relaxation and discuss its basic properties. In this second part, 
our first result states that given a feasible initial point satisfying the linear independence constraint qualification (LICQ) condition, the penalized relaxation produces a unique solution which is feasible for the original QCQP and its objective value is not worse than that of the initial point. 
Our second result states that if the initial point is infeasible, but sufficiently close to the feasible set, and satisfies a generalized LICQ condition, then there exists a unique optimal solution to the penalized convex relaxation problem which is feasible for the original QCQP. Motivated by these results on constructing feasible points, we propose a sequential procedure for solving non-convex QCQPs with convergence guarantees to a point which satisfies Karush–Kuhn–Tucker (KKT) optimality conditions.
Finally, we demonstrate its performance on benchmark instances from QPLIB library \cite{FuriniEtAl2017TR} as well as on large-scale system identification problems.


\subsection{Notations}

For a given vector $\abf$ and a matrix $\Abf$, the symbols $a_i$ and $A_{ij}$, respectively, indicate the $i^{th}$ element of $\abf$ and the $(i,j)^{\mathrm{th}}$ element of $\Abf$.
The symbols $\Rbb$, $\Rbb^{n}$, and $\Rbb^{n\times m}$ denote the sets of real scalars, real vectors of size $n$, and real matrices of size $n\times m$, respectively. The set of $n\times n$ real symmetric matrices is shown by $\Sbb_n$. 
The notations $\Dbb_n$ and $\Sbb_n^{+}$, respectively, represent the sets of diagonally-dominant and positive semidefinite members of $\Sbb_n$. The $n\times n$ identity matrix is denoted by $\Ibf_n$. The origins of $\Rbb^n$ and $\Rbb^{n\times m}$ are denoted by $\boldsymbol{0}_n$ and $\boldsymbol{0}_{n\times m}$, respectively, while $\boldsymbol{1}_n$ denotes the all one $n$-dimensional vector. For a pair of $n\times n$ symmetric matrices $(\Abf,\Bbf)$, the notation $\Abf\succeq\Bbf$ means that $\Abf-\Bbf$ is positive semidefinite, whereas $\Abf\succ\Bbf$ means that $\Abf-\Bbf$ is positive definite. Given a matrix $\Abf\in\mathbb{R}^{m\times n}$, the notation $\sigma_{\min}(\Abf)$ represents the minimum singular value of $\Abf$. 
The symbols $\langle\cdot\,,\cdot\rangle$ and $\|\cdot\|_{\Frm}$ denote the Frobenius inner product and norm of matrices, respectively. Kronecker product of matrices is represented by $\otimes$. The notations $\|\cdot\|_1$ and $\|\cdot\|_2$ denote the $1$-  and $2$-norms of vectors/matrices, respectively. The superscript $(\cdot)^\top$ and the symbol $\mathrm{tr}\{\cdot\}$ represent the transpose and trace operators, respectively. $\mathrm{vec}\{\cdot\}$ denotes the vectorization operator. The notation $|\cdot|$ represents either the absolute value operator or cardinality of a set, depending on the context. For every $\xbf\in\Rbb^n$, the notation $[\xbf]$ represents an $n\times n$ diagonal matrix with the elements of $\xbf$. The notations $\nabla f(\abf)$ and $\nabla^2f(\abf)$, respectively, represent the gradient and Hessian of the function $f$ at a point $\abf$. The boundary of a subset $S$ of a topological space $X$ is defined as the set of points in the closure of $S$ not belonging to the interior of $S$. 

\section{Theoretical Results}\label{sec_thm}

This section formally states the results on the objective penalization and the convergence of the proposed sequential algorithm. We first need to define the notions of {\it feasibility distance}, {\it quasi-binding constraints}, {\it generalized LICQ}, {\it singularity function}, and {\it pencil norm}.

\subsection{Preliminaries} 
As discussed in \cite{parabolic_part1}, we consider the problem of finding a matrix $\Ybf\in\Rbb^{n\times m}$ that minimizes a quadratic objective function subject to a set of equality constraints $\Ecal$ and inequality constraints $\Ical$, i.e., 
\begin{subequations}
	\begin{align}
		&\underset{\Ybf\in\Rbb^{n\times m}}{\text{minimize}}&& 
		\!\!\!q\Tp_0(\Ybf)  \label{QCQP_tensor_a}\\
		&\text{subject to} && 
		\!\!\!q\Tp_k(\Ybf) = 0, &&&&\!\!\!\!\!\!k\in\Ecal, \label{QCQP_tensor_b}\\
		& && 
		\!\!\!q\Tp_k(\Ybf) \leq 0, &&&&\!\!\!\!\!\!k\in\Ical, \label{QCQP_tensor_c}
	\end{align}
\end{subequations}
where each function $q_k:\Rbb^{n\times m}\to\Rbb$ is defined as $q\Tp_k(\Ybf)\triangleq\tr\{\Ybf\T\!\Abf\Tp_k\Ybf\}+2\,\tr\{\Bbf_k\T \Ybf\}+c\Tp_{k}$, and $\{\Abf_k\in\Sbb_{n},\;\Bbf_k\in\Rbb^{n\times m},\;c_k\in\Rbb\}_{k\in\{0\}\cup\Ecal\cup\Ical}$ are given matrices/scalars. 
We propose the following penalized convex relaxation for problem \eqref{QCQP_tensor_a} -- \eqref{QCQP_tensor_c}:
\begin{subequations}
	\begin{align}
		&\underset{\begin{subarray}{l}\Ybf\in\Rbb^{n\times m}\\\hspace{-0.2mm}\Xbf\in\Sbb_n\end{subarray}}
		{\text{minimize}}&& 
		\!\!\!\bar{q}\Tp_0(\Ybf\!,\Xbf) +\eta\times \tr\{\Xbf-2\Ych\Ybf\T+\Ych\Ych\T\}\label{prob_relax_pen_1a}\\
		&\text{subject to} && 
		\!\!\!\bar{q}\Tp_k(\Ybf\!,\Xbf) = 0,&&&&\!\!\!\!\!\!k\in\Ecal,\label{prob_relax_pen_1b}\\
		& && 
		\!\!\!\bar{q}\Tp_k(\Ybf\!,\Xbf) \leq 0,&&&&\!\!\!\!\!\!k\in\Ical,\label{prob_relax_pen_1c}\\
		& && 
		\!\!\!X_{ii}+ X_{jj}-2 X_{ij}\geq \| (\ebf_i-\ebf_j)\T \Ybf\|^2_2,&&&&
		\!\!\!\!\!\! i,j\in\Ncal,\label{prob_relax_pen_1d}\\
		& && 
		\!\!\!X_{ii}+ X_{jj}+2 X_{ij}\geq \| (\ebf_i+\ebf_j)\T \Ybf\|^2_2,&&&&
		\!\!\!\!\!\! i,j\in\Ncal.\label{prob_relax_pen_1e}
	\end{align}
\end{subequations}
where $\bar{q}\Tp_k(\Ybf\!,\Xbf)\triangleq\tr\{\Abf\Tp_k\Xbf\}+2\,\tr\{\Bbf_k\T \Ybf\}+c\Tp_{k}$ for each $k\in\{0\}\cup\Ecal\cup\Ical$, $\Ych\in\Rbb^{n\times m}$ is an initial point, and $\eta>0$ is a constant. One can adopt a sequential framework to obtain feasible and near optimal points for  problem \eqref{QCQP_tensor_a} -- \eqref{QCQP_tensor_c}. This procedure is detailed in Algorithm \ref{alg:1}.

The following definition offers a distance measure from the feasible set of this problem.
\begin{definition}[Feasibility Distance]
	Denote $\Fcal\subseteq\Rbb^{n\times m}$ as the feasible set for  problem \eqref{QCQP_tensor_a} -- \eqref{QCQP_tensor_c}. The feasibility distance function $d_{\Fcal}:\Rbb^{n\times m}\to \Rbb$ is defined as:
	\begin{align}
		d_{\Fcal}(\Ych)\triangleq \min\{\|\Ybf-\Ych\|_{\Frm}\,|\,\Ybf\in\Fcal\}.\label{eq:feas}
	\end{align}
	\label{def:feas}\end{definition}

Given a feasible point $\Ybf\in\Fcal$, one can define binding constraints in order to form the Jacobian matrix at point $\Ybf$. The next definition extends this to points that are not necessarily feasible.
\begin{definition}[Quasi-binding Constraints] For every $\Ych\in\Rbb^{n\times m}$, define the set of quasi-binding constraints as:
	\begin{align}
		\Bcal_{\Ych}\triangleq\Ecal\cup\{k\in\Ical\;\big|\;\tilde{q}_k(\Ych)\geq 0\},
	\end{align}
	where, for each $k\in\{0\}\cup\Ecal\cup\Ical$, the expanded function $\tilde{q}_k:\Rbb^{n\times m}\to\Rbb$ is defined as:
	\begin{align}
		\tilde{q}_k(\Ych)\triangleq q_k(\Ych)+ \|\nabla q_k(\Ych)\|_{\Frm}\,
		d_{\Fcal}(\Ych)+
		\|\Abf_k\|\Tp_{2}\, d_{\Fcal}(\Ych)^2.\label{eq:binding}
	\end{align}	
	Observe that if $\Ych$ is feasible, then $\Bcal_{\Ych}$ is the set of binding constraints and $\tilde{q}_k(\Ych)=q_k(\Ych)$, for every $k\in\{0\}\cup\Ecal\cup\Ical$.
	\label{def:binding}\end{definition}

\begin{algorithm}[t]
	\caption{Sequential Penalized Parabolic Relaxation.}\label{alg:1}
	\begin{algorithmic}[1]
		\Require {$\Ybf^{(0)}\in\Rbb^{n\times m}$, $\eta>0$ and $l:=0$}
		\Repeat
		\State $l:=l+1$
		\State solve the penalized convex problem \eqref{prob_relax_pen_1a} -- \eqref{prob_relax_pen_1e} with $\Ych=\Ybf^{(l-1)}$ to obtain $\big(\Yst,\Xst\big)$
		\State $\Ybf^{(l)}:=\Yst$
		\Until {stopping criterion is met.}
		\Ensure {$\Ybf^{(l)}$}
	\end{algorithmic}\label{al:alg_1}
\end{algorithm}

The next definition offers a generalization for the notion of LICQ regularity as well as singularity of any given point in $\Rbb^{n\times m}$.
\begin{definition}[Generalized LICQ]
	For every $\Ych\in\Rbb^{n\times m}$ and $\Qcal\triangleq\{k_1,k_2,\ldots,k_{|\Ccal|}\}\subseteq\Ecal\cup\Ical$, define the Jacobian matrix as:  
	\begin{align}
		\Jbf_{\!\Qcal}(\Ych)\triangleq 
		\Big[\mathrm{vec}\big\{\nabla q\Tp_{k_1}(\Ych)\big\}, \mathrm{vec}\big\{\nabla q\Tp_{k_2}(\Ych)\big\}, 
		\ldots, \mathrm{vec}\big\{\nabla q\Tp_{k_{|\Ccal|}}(\Ych)\big\}\Big]\T.
	\end{align}
	The point $\Ych$ is said to satisfy GLICQ condition if the rows of $\Jbf_{\!\Bcal_{\Ych}}(\Ych)$  are linearly independent, where $\Bcal_{\Ych}$ is the set of quasi-binding constraints for $\Ych$. Moreover, the {\bf singularity function} $s:\Rbb^{n\times m}\to\Rbb$ is defined as: 
	\begin{align}
		s(\Ych)\triangleq
		\left\{\begin{matrix*}[l]
			\sigma_{\min}\big\{\Jbf_{\!\Bcal_{\Ych}}(\Ych)\big\}, &\qquad |\Bcal_{\Ych}|\leq n,\\ 
			0, & \qquad\text{otherwise,}
		\end{matrix*}\right.
	\end{align}
	where $\sigma_{\min}$ denotes the smallest singular value operator.
	\label{def:GLICQ}\end{definition}
In general, it is computationally hard to calculate the exact distance to $\Fcal$ and to verify GLICQ as a consequence. However, local search methods can be used in practice to find a local solution for \eqref{eq:feas}, resulting in upper bounds on the distance to $\Fcal$. The next definition introduces the notion of matrix pencil corresponding to problem \eqref{QCQP_tensor_a} -- \eqref{QCQP_tensor_c} as a measure for the intensity of constraint matrices. 
\begin{definition}[Pencil Norm] 
	For every $\Qcal\subseteq\Ecal\cup\Ical$, define the matrix function $p_{\Qcal}\Tp:\Rbb^{|\Qcal|}\to\Sbb_n$ as:
	\begin{align}
		p_{\Qcal}\Tp(\taubf)\triangleq\sum_{k\in\Qcal}\tau_k\Abf_k.
	\end{align}
	Moreover, for every $k\geq 0$ define the pencil norm $\rho_k$ as:
	\begin{align}
		\rho_k\triangleq\max
		\left\{\|p_{\Ecal\cup\Ical}\Tp(\taubf)\|\Tp_{k}\;\big|\;
		\taubf\in\Rbb^{|\Ecal\cup\Ical|}\;\wedge\;
		\|\taubf\|_2=1\right\},
	\end{align}
	which is upperbounded by $\sum_{k\in\Ecal\cup\Ical}{\|\Abf_k\|\Tp_k}$.
	\label{def:pencil}\end{definition}

\subsection{Statement of Theorems} 

The next theorem states that if $\eta$ is sufficiently large, then penalization preserves the feasibility of the initial point.

\begin{theorem}\label{thm_feas}
	Let $\Ych\in\Fcal$ be an LICQ feasible point for problem \eqref{QCQP_tensor_a} -- \eqref{QCQP_tensor_c}. If
	\begin{subequations}
		\begin{align}
			&\!\!\!\!\!\eta > \|\Abf_0\|\Tp_1+\|\Abf_0\|\Tp_2+ \frac
{
	2(2\rho_1+\rho_2)\|\nabla q_0(\Ych)\|_{\Frm}
}
{s(\Ych)},\\
			&\!\!\!\!\!\eta>\|\Abf_0\|\Tp_2+\|\nabla q_0(\Ych)\|_{\Frm}\left(
			\sqrt{\frac{\|\Abf_k\|\Tp_2}{|q_k(\Ych)|}}+
			\frac{\|\nabla q_k(\Ych)\|_{\Frm}}{|q_k(\Ych)|}\right),
			\quad\forall k\in\Ical\setminus\Bcal_{\Ych},
		\end{align}
	\end{subequations}
	where $\Bcal_{\Ych}$ is the set of binding constraints for $\Ych$, then the convex problem \eqref{prob_relax_pen_1a} -- \eqref{prob_relax_pen_1e} has a unique solution $(\Yst,\Xst)$ that satisfies $\Xst=\Yst\Yst\T$ and $q_0(\Yst)\leq q_0(\Ych)$. 
\end{theorem}
\begin{proof}
	Refer to Section \ref{sec_proofs} for the proof.
	\qed\end{proof}

The next theorem states that if $\Ych$ is not feasible but close to $\Fcal$, then penalization results in a feasible point as well.

\begin{theorem} \label{thm_nfeas}
	Let $\Ych\in\Rbb^{n\times m}$ be a GLICQ point for problem \eqref{QCQP_tensor_a} -- \eqref{QCQP_tensor_c} that satisfies:
	\begin{align}
		s(\Ych)> 2(\rho_1+\rho_2) d_{\Fcal}(\Ych).
	\end{align}
If
	\begin{subequations}
		\begin{align}
			&\eta > \|\Abf_0\|\Tp_1+\|\Abf_0\|\Tp_2+\nonumber\\
			&\quad\quad \frac
			{
				2\rho_1\|\Abf_0\|\Tp_1 d_{\Fcal}(\Ych)+
				2(2\rho_1+\rho_2)(\|\nabla q_0(\Ych)\|_{\Frm}+\|\Abf_0\|\Tp_2 d_{\Fcal}(\Ych))
				}
			{s(\Ych)- 2(\rho_1+\rho_2) d_{\Fcal}(\Ych)}, \label{thm_nfeas_cond1}\\
			&\eta>\|\Abf_0\|\Tp_2+\left( \|\nabla q_0(\Ych)\|_{\Frm}+\|\Abf_0\|\Tp_2 d_{\Fcal}(\Ych)\right)  \times \nonumber\\
			&
			\quad\quad \left(
			\sqrt{\frac{\|\Abf_k\|\Tp_2}{|\tilde{q}_k(\Ych)|}}+
			\frac{\|\nabla q_k(\Ych)\|_{\Frm}+2\|\Abf_k\|\Tp_2 d_{\Fcal}(\Ych)}{|\tilde{q}_k(\Ych)|^{\phantom{\frac{1}{2}}}}
			\right),
			\quad\forall k\in\Ical\setminus\Bcal_{\Ych}, \label{thm_nfeas_cond2}
		\end{align}
	\end{subequations}
	where $\Bcal_{\Ych}$ is the set of quasi-binding constraints for $\Ych$ and functions $\{\tilde{q}_k\}_{k\in\Ical}$ are defined by equation \eqref{eq:binding}, then the convex problem \eqref{prob_relax_pen_1a} -- \eqref{prob_relax_pen_1e} has a unique solution $(\Yst,\Xst)$ that satisfies $\Xst=\Yst\Yst\T$ and $q_0(\Yst)\leq \tilde{q}_0(\Ych)$.
\end{theorem}
\begin{proof}
	Refer to Section \ref{sec_proofs} for the proof.
	\qed\end{proof}

The next theorem offers a convergence guarantee for Algorithm \ref{alg:1}.

\begin{theorem}\label{thm_conv}
	Let $\Fch\triangleq\big\{\Ybf\in\Fcal\;|\;q_0(\Ybf)\leq\qch\big\}$ denote an epigraph of problem \eqref{QCQP_tensor_a} -- \eqref{QCQP_tensor_c} such that $s(\Ybf_1)^{-1}\|\nabla q_0(\Ybf_2)\|_{\Frm}$ is bounded for every $\Ybf_1,\Ybf_2\in\Fch$. If $\Ybf^{0}\in\Fch$, and
	\begin{align}
		\eta > \|\Abf\|_1+\|\Abf\|_2 +
	3\rho_1\,\frac{\max_{\Ybf\in\Fch}\{\|\nabla q_0(\Ybf)\|_{\Frm}\}}
		{\min_{\Ybf\in\Fch}\{s(\Ybf)\}},\label{asmp_conv}
	\end{align}
	then every member of the sequence generated by Algorithm \ref{alg:1} satisfies $\Xst^{(l)}=$\linebreak $\Yst^{(l)}(\Yst^{(l)})\T$. Moreover, the sequence $\{\Yst^{(l)}\}^{\infty}_{l=0}$ converges to a point that satisfies the Karush–Kuhn–Tucker optimality conditions for problem \eqref{QCQP_tensor_a} -- \eqref{QCQP_tensor_c}, with non-increasing objective values .
\end{theorem}
\begin{proof}
	Refer to Section \ref{sec_proofs} for the proofs.
	\qed\end{proof}

\subsection{Nesterov's Acceleration} 

In this section, we explore the possibility of applying Nesterov's acceleration scheme to Algorithm \ref{alg:1} as a heuristic approach. As delineated in Algorithm \eqref{alg:2}, the point $\big(\Yst^{(l)},\Xst^{(l)}\big)$ denotes the solution to problem \eqref{prob_relax_pen_1a} -- \eqref{prob_relax_pen_1e} with initial point $\Ych=\Ych^{(l)}$ and penalty term $\eta_{l}\Tp\!$, where:
\begin{align}
	\Ych^{(l)}:=(1-\lambda_{l})\Yst^{(l-1)}+\lambda_{l}\Ych^{(l-1)},
\end{align}
and the values $\{\lambda_{l}\Tp\!\in[0,1)\}^{\infty}_{l=1}$ and $\{\eta_{l}\Tp\!>0\}^{\infty}_{l=1}$ are determined dynamically at each step. Due to non-convexity of $\Fcal$, even if $\Yst^{(l-1)}$ is feasible, the point $\Ych^{(l)}$ may not belong to $\Fcal$ and Theorem \ref{thm_feas} cannot be used to ensure the feasibility of $\Yst^{(l)}$. 

\begin{algorithm}[b]
	\caption{Accelerated Sequential Heuristic}\label{alg:2}
	\begin{algorithmic}[1]
		\Require {$\Ych^{(0)}\in\Rbb^{n\times m}$, $\lambda_{0}: = 1$ and $\ell:=0$}
		\Repeat
		\State $\ell:=\ell+1$
		\State select $\eta_{\ell}\Tp\!>0$ and $\lambda_{l}\Tp\!\in[0,1)$
		\State $\Ych^{(\ell)}:=(1-\lambda_{\ell})\Yst^{(\ell-1)}+\lambda_{\ell}\Ych^{(\ell-1)}$
		\State solve the penalized convex problem \eqref{prob_relax_pen_1a} -- \eqref{prob_relax_pen_1e} with $\Ych=\Ych^{(\ell)}$ and $\eta=\eta\Tp_{\ell}$ to obtain $\big(\Yst^{(\ell)},\Xst^{(\ell)}\big)$
		\Until {stopping criterion is met.}
		\Ensure {$\Yst^{(\ell)}$}
	\end{algorithmic}\label{al:alg_2}
\end{algorithm}

Despite this issue, we show that if $\Yst^{(l-1)}$ is feasible and satisfies LICQ, then it is possible to select the values $\eta\Tp_{l}$ and $\lambda_{l}$ in step 3 of Algorithm \eqref{alg:2} such that the conditions of Theorem \ref{thm_nfeas} are satisfied, 
which guarantees that $\Yst^{(l)}\in\Fcal$. To this end, assume that:
	\begin{subequations}
	\begin{align}
		&\eta_l > \|\Abf_0\|\Tp_1+\|\Abf_0\|\Tp_2+2(2\rho_1+\rho_2)\,\frac
		{\|\nabla q_0(\Yst^{(l-1)})\|_{\Frm}}
		{s(\Yst^{(l-1)})},\\
		&\eta_l>\|\Abf_0\|\Tp_2+\|\nabla q_0(\Yst^{(l-1)})\|_{\Frm}\times \nonumber\\
		&
		\quad\quad \left(
		\sqrt{\frac{\|\Abf_k\|\Tp_2}{|q_k(\Yst^{(l-1)})|}}+
		\frac{\|\nabla q_k(\Yst^{(l-1)})\|_{\Frm}}{|q_k(\Yst^{(l-1)})|^{\phantom{\frac{1}{2}}}}
		\right),
		\quad\forall k\in\Ical\setminus\Bcal_{\Yst^{(l-1)}},
	\end{align}
\end{subequations}
Now, since
\begin{align}
	d_\Fcal\big(\Ych^{(l)}\big)\!=\!d_\Fcal\big(\Yst^{(l-1)}+\lambda_{l}(\Ych^{(l-1)}-\Yst^{(l-1)})\big)\leq\lambda_{l}\|\Ych^{(l-1)}-\Yst^{(l-1)}\|_{\Frm},
\end{align}
if $\lambda_{l}$ is sufficiently small,  then $\Ych^{(l)}$ can get arbitrarily close to $\Yst^{(l-1)}$ and, as a consequence, 
we have: 
\begin{subequations}
	\begin{align}
		\Bcal_{\Ych^{(l)}} &\subseteq \Bcal_{\Yst^{(l-1)}},\nonumber\\
		s\big(\Ych^{(l)}\big) &> 4\rho\; d_{\Fcal}\big(\Ych^{(l)}\big),\nonumber\\
			\eta\Tp_{l}\! &>\! \|\Abf_0\|\Tp_1+\|\Abf_0\|\Tp_2+\nonumber\\
&\frac
{
	2\rho_1\|\Abf_0\|\Tp_1 d_{\Fcal}\big(\Ych^{(l)}\big)+
	2(2\rho_1+\rho_2)\big[\|\nabla q_0\big(\Ych^{(l)}\big)\|_{\Frm}+\|\Abf_0\|\Tp_2 d_{\Fcal}\big(\Ych^{(l)}\big)\big]
}
{s\big(\Ych^{(l)}\big)- 2(\rho_1+\rho_2) d_{\Fcal}\big(\Ych^{(l)}\big)},\nonumber\\
		\eta\Tp_{l}\! &>\!\|\Abf_0\|\Tp_2\!+\!\left( \|\nabla q_0\big(\Ych^{(l)}\big)\|_{\Frm}\!+\|\Abf_0\|\Tp_2 d_{\Fcal}\big(\Ych^{(l)}\big)\right) \! \times \nonumber\\
		&  \left(\!
		\sqrt{\frac{\|\Abf_k\|\Tp_2}{|\tilde{q}_k\big(\Ych^{(l)}\big)|}}\!+
		\frac{\|\nabla q_k\big(\Ych^{(l)}\big)\|_{\Frm}\!+\!2\|\Abf_k\|\Tp_2 d_{\Fcal}\big(\Ych^{(l)}\big)}{|\tilde{q}_k\big(\Ych^{(l)}\big)|^{\phantom{\frac{1}{2}}}}
		\!\right),
		\quad\forall k\!\in\!\Ical\!\setminus\!\Bcal_{\Ych^{(l)}},\nonumber
	\end{align}
\end{subequations}
where $\Bcal_{\Ych^{(l)}}$ is the set of quasi-binding constraints for $\Ych^{(l)}$. Therefore, according to Theorem \ref{thm_nfeas}, we have $\Yst^{(l)}\in\Fcal$. This way, one may ensure the feasibility of the sequence $\{\Yst^{(l)}\}^\infty_{l=1}$. We leave the convergence and theoretical analysis of this heuristic for future work.

\subsection{Proofs}\label{sec_proofs}

In order to prove Theorems \ref{thm_feas}, \ref{thm_nfeas},  and \ref{thm_conv}, we consider the following auxiliary optimization problem:
\begin{subequations}
	\begin{align}
		&\underset{\Ybf\in\Rbb^{n\times m}}{\text{minimize}}&& 
		\!\!\!q\Tp_0(\Ybf)+\eta\|\Ybf-\Ych\|_{\Frm}^2  \label{aux_1a}\\
		&\text{subject to} && 
		\!\!\!q\Tp_k(\Ybf) = 0,&&&&\!\!\!\!\!\!k\in\Ecal,\label{aux_1b}\\
		& && 
		\!\!\!q\Tp_k(\Ybf) \leq 0,&&&&\!\!\!\!\!\!k\in\Ical.\label{aux_1c}
	\end{align}
\end{subequations}
Observe that the convex problem \eqref{prob_relax_pen_1a} -- \eqref{prob_relax_pen_1e} is a relaxation of \eqref{aux_1a} -- \eqref{aux_1c} and this is the motivation for introducing \eqref{aux_1a} -- \eqref{aux_1c}.  

The next two lemmas show that by increasing the penalty term $\eta$, the optimal solution $\Yst$ can get as close to the initial point $\Ych$ as $d_{\Fcal}(\Ych)$. This lemma will later be used to
show that $\Yst$ can inherit GLICQ property from $\Ych$.
\begin{lemma}\label{lm:q_bound}
	The following inequality holds for any $k\in\{0\}\cup\Ecal\cup\Ical$ and arbitrary $\Ybf_{\!\!1},\Ybf_{\!\!2}\in\Rbb^{n\times m}$:
	\begin{subequations}
		\begin{align}
			&|q_k(\Ybf_{\!\!1})-q_k(\Ybf_{\!\!2})|\leq 
			\|\Abf_k\|\Tp_2\|\Ybf_{\!\!1}-\Ybf_{\!\!2}\|_{\Frm}^2+
			\|\nabla q_k(\Ybf_{\!\!2})\|_{\Frm}\|\Ybf_{\!\!1}-\Ybf_{\!\!2}\|_{\Frm},
			\label{lm1:eq1}\\
			&|\|\nabla q_k(\Ybf_{\!\!1})\|_{\Frm} - \|\nabla q_k(\Ybf_{\!\!2})\|_{\Frm}|\leq 2\|\Abf_k\|\Tp_2\|\Ybf_{\!\!1}-\Ybf_{\!\!2}\|_{\Frm}.
			\label{lm1:eq2}
		\end{align}
	\end{subequations}
\end{lemma}
\begin{proof}
	The proof of \eqref{lm1:eq2} is straightforward. Additionally:
	\begin{subequations}
		\begin{align}
			\!\!\!\!\!\!|q_k(\Ybf_{\!\!1}) - &q_k(\Ybf_{\!\!2})|\!=\! |\la\Abf\Tp_k,(\Ybf_{\!\!1}-\Ybf_{\!\!2})(\Ybf_{\!\!1}-\Ybf_{\!\!2})\T\ra
			\!+\!2\la\Abf\Tp_k\Ybf_{\!\!2}\!+\!\Bbf_k\Tp,\Ybf_{\!\!1}-\Ybf_{\!\!2}\ra|\\
			&\!\leq\! |\la\Abf\Tp_k,(\Ybf_{\!\!1}-\Ybf_{\!\!2})(\Ybf_{\!\!1}-\Ybf_{\!\!2})\T\ra|
			\!+\!2|\la\Abf\Tp_k\Ybf_{\!\!2}\!+\!\Bbf_k\Tp,\Ybf_{\!\!1}-\Ybf_{\!\!2}\ra|\\
			&\!\leq\!\|\Abf_k\|\Tp_2\|\Ybf_{\!\!1}-\Ybf_{\!\!2}\|_{\Frm}^2+
			2\|\Abf\Tp_k\Ybf_{\!\!2}+\Bbf_k\Tp\|_{\Frm}\|\Ybf_{\!\!1}-\Ybf_{\!\!2}\|_{\Frm}\\
			&\!\leq\!\|\Abf_k\|\Tp_{2}\|\Ybf_{\!\!1}-\Ybf_{\!\!2}\|_{\Frm}^2+
			\|\nabla q_k(\Ybf_{\!\!2})\|_{\Frm}\|\Ybf_{\!\!1}-\Ybf_{\!\!2}\|_{\Frm},
		\end{align}
	\end{subequations}
	which proves \eqref{lm1:eq1}.
	\qed\end{proof}

\begin{lemma}\label{lm:d_bound}
	If $\eta>\|\Abf_0\|\Tp_2$, then every optimal solution $\Yst$ of problem \eqref{aux_1a} -- \eqref{aux_1b} satisfies:
	\begin{align}
		0\leq\|\Yst-\Ych\|_{\Frm}-d_{\Fcal}(\Ych)\leq \frac{\|\Abf_0\|\Tp_2 d_{\Fcal}(\Ych)+\|\nabla q_0(\Ych)\|_{\Frm} }{\eta-\|\Abf_0\|\Tp_2},\label{bound_d}
	\end{align}
	where $d_{\Fcal}$ is defined by equation \eqref{eq:feas}.
\end{lemma}
\begin{proof}
	According to Definition \ref{def:feas}, the distance between an arbitrary point $\Ych$ and any point in $\Fcal$ is greater than or equal to $d_{\Fcal}(\Ych)$. This implies that $\|\Yst-\Ych\|_{\Frm}-d_{\Fcal}(\Ych)$ is lower bounded by zero. To prove the upper bound, let $\bar{\Ybf}$ be an arbitrary member of $\{\Ybf\in\Fcal\,|\,\|\Ybf-\Yst\|_{\Frm}=d_{\Fcal}(\Yst)\}$. Since $\Yst$ is an optimal solution to problem \eqref{aux_1a} -- \eqref{aux_1b}, one can write:
	\begin{subequations}
		\begin{align}
			-\|\nabla q_0(\Ych)\|_{\Frm}
			&\|\Yst-\Ych\|_{\Frm}+\big(\eta-\|\Abf_0\|\Tp_2\big)\|\Yst-\Ych\|_{\Frm}^2\\
			\leq& -q_0(\Ych)+q_0(\Yst)+\eta\|\Yst-\Ych\|_{\Frm}^2\label{eq_eq_1}\\
			\leq& -q_0(\Ych)+q_0(\bar{\Ybf})+\eta\|\bar{\Ybf}-\Ych\|_{\Frm}^2\\
			\leq& \|\nabla q_0(\Ych)\|_{\Frm}\|\bar{\Ybf}-\Ych\|_{\Frm}
			+\big(\eta+\|\Abf_0\|\Tp_2\big)\|\bar{\Ybf}-\Ych\|_{\Frm}^2\!\!\!\!\label{eq_eq_2}\\
			=& \|\nabla q_0(\Ych)\|_{\Frm} d_{\Fcal}(\Ych)+\big(\eta+\|\Abf_0\|\Tp_2\big)d_{\Fcal}(\Ych)^2,
		\end{align}
	\end{subequations}
	where \eqref{eq_eq_1} and \eqref{eq_eq_2} are concluded form Lemma \ref{lm:q_bound}. Hence:
\begin{align}
&-\|\nabla q_0(\Ych)\|_{\Frm} \big(\|\Yst-\Ych\|_{\Frm}+ d_{\Fcal}(\Ych)\big)
+\big(\eta-\|\Abf_0\|\Tp_2\big)\big(\|\Yst-\Ych\|_{\Frm}^2-d_{\Fcal}(\Ych)^2\big)\nonumber\\
&\leq 2\|\Abf_0\|\Tp_2 d_{\Fcal}(\Ych)^2,
\end{align}	
and
\begin{align}
			&\left(\|\Yst-\Ych\|_{\Frm}+d_{\Fcal}(\Ych)\right)
			\left[\left(\|\Yst-\Ych\|_{\Frm}-d_{\Fcal}(\Ych)\right)\big(\eta-\|\Abf_0\|\Tp_2\big)-
			\|\nabla q_0(\Ych)\|_{\Frm}\right] \nonumber\\
			&\leq 2\|\Abf_0\|\Tp_2 d_{\Fcal}(\Ych)^2.
					\end{align}
Now due to feasibility of $\Yst$  we have $d_{\Fcal}(\Ych)\leq \|\Yst-\Ych\|_{\Frm}$ which implies: \pagebreak
			\begin{align}
			&\left(\|\Yst-\Ych\|_{\Frm}-d_{\Fcal}(\Ych)\right)\big(\eta-\|\Abf_0\|\Tp_2\big)-
			\|\nabla q_0(\Ych)\|_{\Frm} \nonumber\\ 
			&\leq \|\Abf_0\|\Tp_2 d_{\Fcal}(\Ych)\times\frac{2 d_{\Fcal}(\Ych)}{\|\Yst-\Ych\|_{\Frm}+d_{\Fcal}(\Ych)} \leq \|\Abf_0\|\Tp_2 d_{\Fcal}(\Ych),
		\end{align}
	which concludes the right side of \eqref{bound_d}.
	\qed\end{proof}

The next lemma shows that if $\Ych$ satisfies GLICQ and $\eta$ is sufficiently large, then $\Yst$ satisfies GLICQ as well.
\begin{lemma}\label{lm:GLICQ} Let $\Ych\in\Rbb^{n\times m}$ be a GLICQ point that satisfies 
	$s(\Ych) > 2\rho\, d_{\Fcal}(\Ych)$ and assume that:
	\begin{subequations}
		\begin{align}	
			&\eta>\|\Abf_0\|\Tp_2+2\rho_2\times
			\frac{\|\nabla q_0(\Ych)\|_{\Frm}+\|\Abf_0\|\Tp_2 d_{\Fcal}(\Ych)}
			{s(\Ych) - 2\rho_2\, d_{\Fcal}(\Ych)},\label{qs_bound}\\
			&\eta>\|\Abf_0\|\Tp_2+\left(\|\nabla q_0(\Ych)\|_{\Frm}+\|\Abf_0\|\Tp_2 d_{\Fcal}(\Ych)\right) \times \nonumber\\
			& \quad\quad \left(
			\sqrt{\frac{\|\Abf_k\|\Tp_2}{|\tilde{q}_k(\Ych)|}}+
			\frac{\|\nabla q_k(\Ych)\|_{\Frm}+2\|\Abf_k\|\Tp_2 d_{\Fcal}(\Ych)}{|\tilde{q}_k(\Ych)|^{\phantom{\frac{1}{2}}}}
			\right),
			\qquad\forall k\in\Ical\setminus\Bcal_{\Ych}, \label{q_bound}
		\end{align}
	\end{subequations}
	where $\Bcal_{\Ych}$ denotes the set of quasi-binding constraints for $\Ych$. Then, every solution $\Yst$ of problem \eqref{aux_1a} -- \eqref{aux_1b} is LICQ and satisfies:
	\begin{align}
		&s(\Yst)\geq s(\Ych) - 2\rho_2\,\|\Ych-\Yst\|_{\Frm},\label{sss_bound}
	\end{align}
	where functions $\{\tilde{q}_k\}_{k\in\Ical}$ are defined in \eqref{eq:binding} and $\Bcal_{\Ych}$ denotes the set of quasi-binding constraints for $\Ych$.
\end{lemma}
\begin{proof}
	We first need to show that $\Bcal_{\Yst}\subseteq\Bcal_{\Ych}$, where $\Bcal_{\Yst}$ denotes the set of quasi-binding constraints for $\Yst$. Let $k\in\Ical\setminus\Bcal_{\Ych}$. Then, according to assumption \eqref{q_bound} and Lemma \ref{lm:d_bound}, we have: 
	\begin{subequations}
	\begin{align}
		\|\Ych-\Yst\|_{\Frm}-&d_{\Fcal}(\Ych) 
		\leq
		\frac{\|\Abf_0\|\Tp_2 d_{\Fcal}(\Ych) + \|\nabla q_0(\Ych)\|_{\Frm} }{\eta-\|\Abf_0\|\Tp_2}\\
		<& \frac{|\tilde{q}_k(\Ych)|}
		{\|\nabla q_k(\Ych)\|_{\Frm}+2\|\Abf_k\|\Tp_2 d_{\Fcal}(\Ych)+
			\sqrt{\|\Abf_k\|\Tp_2|\tilde{q}_k(\Ych)|} },
	\end{align}
	\end{subequations}
	where the second line is concluded by substituting the lower bound for $\eta$. Hence: 
	\begin{align}
		q_k(\Yst) \leq& q_k(\Ych)+ \|\nabla q_k(\Ych)\|_{\Frm}\|\Ych-\Yst\|_{\Frm}+\|\Abf_k\|\Tp_2\|\Ych-\Yst\|_{\Frm}^2\nonumber\\
		\leq& q_k(\Ych)\!+\! \|\nabla q_k(\Ych)\|_{\Frm}\!\left(d_{\Fcal}(\Ych)
		+\frac{ \|\Abf_0\|\Tp_2 d_{\Fcal}(\Ych)\!+\!\|\nabla q_0(\Ych)\|_{\Frm} }{\eta-\|\Abf_0\|\Tp_2}\right)
		\nonumber\\
		&\!+\!\|\Abf_k\|\Tp_2\!\left(d_{\Fcal}(\Ych)\!+\!\frac{ \|\Abf_0\|\Tp_2 d_{\Fcal}(\Ych) + \|\nabla q_0(\Ych)\|_{\Frm} }{\eta-\|\Abf_0\|\Tp_2}\right)^{\!\!2}\nonumber\\
		=& \tilde{q}_k(\Ych)\!+\!
		\frac{ \|\Abf_0\|\Tp_2 d_{\Fcal}(\Ych) + \|\nabla q_0(\Ych)\|_{\Frm} }{\eta-\|\Abf_0\|\Tp_2}\nonumber\\
		&\times \left(\|\nabla q_k(\Ych)\|_{\Frm}+2\|\Abf_k\|\Tp_2 d_{\Fcal}(\Ych)
		+\|\Abf_k\|\Tp_2\frac{ \|\Abf_0\|\Tp_2 d_{\Fcal}(\Ych) + \|\nabla q_0(\Ych)\|_{\Frm} }{\eta-\|\Abf_0\|\Tp_2}
		\right) \nonumber\\
		<& \tilde{q}_k(\Ych)\!+\!
		\frac{ |\tilde{q}_k(\Ych)|}
		{\|\nabla q_k(\Ych)\|_{\Frm}\!+\!2\|\Abf_k\|\Tp_2 d_{\Fcal}(\Ych)\!+\!
			\sqrt{\|\Abf_k\|\Tp_2|\tilde{q}_k(\Ych)|}}\nonumber\\
		&\times \left(\|\nabla q_k(\Ych)\|_{\Frm}\!+\!2\|\Abf_k\|\Tp_2 d_{\Fcal}(\Ych)
		\!+\!\frac{\|\Abf_k\|\Tp_2 |\tilde{q}_k(\Ych)|}
		{\sqrt{\|\Abf_k\|\Tp_2|\tilde{q}_k(\Ych)|}}
		\right) \nonumber\\
		=& \tilde{q}_k(\Ych)+|\tilde{q}_k(\Ych)|=0,
	\end{align}
	which implies that $k\notin\Bcal_{\Yst}$, and therefore $\Bcal_{\Yst}\subseteq\Bcal_{\Ych}$.
	
	Finally, the LICQ regularity of $\Yst$ and the lower bound \eqref{sss_bound} on $s(\Yst)$ can be proven as follows:
	\begin{subequations}	
		\begin{align}
			s(\Yst) &= \min\!\big\{\lVert\Jbf_{\!\Bcal_{\Yst}}(\Yst)\T\xibf\|_{2}\,|\,\|\xibf\|_2=1\big\}\\
			&\geq \min\!\big\{\lVert\Jbf_{\!\Bcal_{\Ych}}(\Yst)\T\nubf\|_{2}\,|\,\|\nubf\|_2=1\big\}\\
			&\geq \min\!\big\{\lVert\Jbf_{\!\Bcal_{\Ych}}(\Ych)\T\nubf\|_{2}-\|[\Jbf_{\!\Bcal_{\Ych}}(\Ych)-\Jbf_{\!\Bcal_{\Ych}}(\Yst)]\T\nubf\|_{2}
			\,|\,\|\nubf\|_2=1\big\}\\
			&\geq  \min\!\big\{\lVert\Jbf_{\!\Bcal_{\Ych}}(\Ych)\T\nubf\|_{2}
			\,|\,\|\nubf\|_2=1\big\} \nonumber\\
			& \quad\quad - 2\max\!\big\{\|p_{\Bcal_{\Ych}}(\nubf)(\Ych-\Yst)\|_{\Frm}\,|\,\|\nubf\|_2=1\big\}\\
			&\geq s(\Ych) - 2\max\!\big\{\|p_{\Bcal_{\Ych}}(\nubf)\|_2\,|\,\|\nubf\|_2=1\big\}\|\Ych-\Yst\|_{\Frm}\\
			&\geq s(\Ych) - 2 \rho_2\times \|\Ych-\Yst\|_{\Frm}\\
			&\geq s(\Ych) - 2 \rho_2\times \left(d_{\Fcal}(\Ych)+\frac{\|\Abf_0\|\Tp_2 d_{\Fcal}(\Ych) + \|\nabla q_0(\Ych)\|_{\Frm}  }{\eta-\|\Abf_0\|\Tp_2}\right)>0,
		\end{align}
	\end{subequations}
	where the last inequality is resulted from \eqref{qs_bound}.
	\qed\end{proof}

The next lemma offers a bound on the dual optimal solution of \eqref{aux_1a} -- \eqref{aux_1c}, which will be used for proving the theorems.
\begin{lemma}\label{lm:tau_bound} Let $\Yst$ be an LICQ optimal solution to problem \eqref{aux_1a} -- \eqref{aux_1b}. There exists a vector of Lagrange multipliers $\taust\in\Rbb^{|{\Bcal_{\Yst}}|}$ associated with $\Yst$ that satisfies:
	\begin{align}
		\|\eta^{-1}\taust\|_{2}\leq 
		\frac{2\|\Yst-\Ych\|_{\Frm}+\eta^{-1}\|\nabla q_0(\Yst)\|_{\Frm}}{s(\Yst)}.
	\end{align}
\end{lemma}
\begin{proof}
	Due to LICQ, there exists a vector of Lagrange multipliers $\taust\in\Rbb^{|\Bcal_{\Yst}|}$ associated with $\Yst$ that satisfies KKT stationarity condition:
	\begin{align}
		\eta(\Yst-\Ych)+(\Abf\Tp_0\Yst+\Bbf\Tp_0)
		+{\sum}_{k\in\Bcal_{\Yst}}\;\accentset{\ast}{\tau}_k(\Abf\Tp_k\Yst+\Bbf_k\Tp)=\boldsymbol{0}_{n\times m},
	\end{align}
	and, therefore:
	\begin{align}
		2\mathrm{vec}\{\eta(\Yst-\Ych)+(\Abf_0\Yst+\Bbf_0)\}
		+\Jbf_{\!\Bcal_{\Yst}}(\Yst)\T\taust=\boldsymbol{0}_{n m}.
	\end{align}
	Since $\Jbf_{\!\Bcal_{\Yst}}(\Yst)$ is full-rank, we have:
	\begin{align}
		\taust=-2\Jbf^{+}_{\!\Bcal_{\Yst}}(\Yst)\T\mathrm{vec}\{\eta(\Yst-\Ych)+(\Abf_0\Yst+\Bbf_0)\}.
	\end{align}
	Hence:
	\begin{subequations}
		\begin{align}
			\|\eta^{-1}\taust\|_{2}&\leq 2\|\Jbf^{+}_{\!\Bcal_{\Yst}}(\Yst)\|_{2}\times\|
			\mathrm{vec}\{(\Yst-\Ych)+\eta^{-1}(\Abf_0\Yst+\Bbf_0)\}\|_{2}
			\\
			&= 2s(\Yst)^{-1}\|
			(\Yst-\Ych)+\eta^{-1}(\Abf_0\Yst+\Bbf_0)\|_{\Frm}
			\\
			&\leq s(\Yst)^{-1}
			\big(2\|\Yst-\Ych\|_{\Frm}
			+\eta^{-1}\|\nabla q_0(\Yst)\|_{\Frm}\big),
		\end{align}
	\end{subequations}
	which proves the lemma.
	\qed\end{proof}

The next lemma offers a sufficient condition on the success of penalization.

\begin{lemma}\label{lm:exact} Let $\Yst$ and $\taust\in\Rbb^{|\Bcal_{\Yst}|}$ be a pair of primal and dual optimal solutions for problem \eqref{aux_1a} -- \eqref{aux_1b}. If the matrix:
	\begin{align}
		\Lambdast\triangleq \eta\Ibf_n+\Abf_0+p_{\Bcal_{\Yst}}(\taust),
		\nonumber
	\end{align}
	is diagonally-dominant, then $(\Yst,\Xst)$ is the unique optimal solution for problem \eqref{prob_relax_pen_1a} -- \eqref{prob_relax_pen_1e}, where $\Xst\triangleq\Yst\Yst\T$.
\end{lemma}
\begin{proof}
	Since $\Lambdast$ is diagonally-dominant, it can serve as a Lagrange multiplier matrix corresponding to the constraints \eqref{prob_relax_pen_1d} -- \eqref{prob_relax_pen_1e}. To prove that $(\Yst,\Xst)$ and $(\taust,\Lambdast)$ are primal and dual optimal solutions, it suffices to verify the following KKT conditions: 
	\begin{subequations}
		\begin{align}
			&\text{Stationarity with respect to $\Xbf$:} &&\hspace{0mm}\eta\Ibf_n+\Abf_0+p_{\Bcal_{\Yst}}(\taust)-\Lambdast = \boldsymbol{0}_{n\times n}, \label{kkt_lifted_1}\\
			&\text{Stationarity with respect to $\Ybf$:} &&\hspace{0mm}\Lambdast\Yst  -\eta\Ych+\Bbf\Tp_0
			+{\sum}_{k\in\Bcal_{\Yst}}\;\accentset{\ast}{\tau}_k\Bbf_k\Tp =\boldsymbol{0}_{n\times m},\label{kkt_lifted_2}\\
			&\text{Primal feasibility:} &&\hspace{0mm}\bar{q}_k(\Yst,\Xst) =0, \qquad k\in\Bcal_{\Yst},\label{kkt_lifted_3}\\
			&\text{Complementary slackness:} &&\hspace{0mm}\la\Lambdast,\Xst-\Yst\Yst\T\ra=0.\label{kkt_lifted_4}
		\end{align}
	\end{subequations}
	Stationarity with respect to $\Xbf$ is followed from the definition of $\Lambdast$.
	Additionally, \eqref{kkt_lifted_2} is followed from the definition of $\Lambdast$, as well as the optimality of $\Yst$ for problem \eqref{aux_1a} -- \eqref{aux_1b}. Finally, the remaining conditions
	\eqref{kkt_lifted_3} and
	\eqref{kkt_lifted_4}
	can be concluded from the definition of $\Xst$ and feasibility of $\Yst$ for \eqref{aux_1a} -- \eqref{aux_1b}.
	\qed\end{proof}

\begin{proof}[Theorem \ref{thm_nfeas}]
	Let $\Yst$ be an optimal solution for problem \eqref{aux_1a} -- \eqref{aux_1b}. According to Lemma \ref{lm:GLICQ}, the point $\Yst$ is LICQ. As a result, there exists a vector of dual variables $\taust\in\Rbb^{|\Bcal_{\Yst}|}$ corresponding to $\Yst$, where $\Bcal_{\Yst}\in\Ecal\cup\Ical$ is the set of binding constraints for $\Yst$. Now, according to Lemmas \ref{lm:q_bound}, \ref{lm:d_bound}, \ref{lm:GLICQ}, and \ref{lm:tau_bound}, we have:
	\begin{align}
		\!\!\!\!\!\!\!\!&\frac{\rho_1\times \|\eta^{-1}\taust\|_{2}}{1-\eta^{-1}\|\Abf_0\|_1\Tp}
		\leq\frac{\rho_1}{1-\eta^{-1}\|\Abf_0\|_1\Tp}
		\times\frac{2\|\Yst-\Ych\|_{\Frm}
			+\eta^{-1}\|\nabla q_0(\Yst)\|_{\Frm}}{s(\Yst)}\nonumber\\
		&\leq\frac{\rho_1}{1-\eta^{-1}\|\Abf_0\|_2\Tp}
		\times\frac{2(1+\eta^{-1}\|\Abf_0\|_2\Tp)\|\Yst-\Ych\|_{\Frm}
			+\eta^{-1}\|\nabla q_0(\Ych)\|_{\Frm}}{s(\Ych) - 2\rho_2\times \|\Ych-\Yst\|_{\Frm}}\nonumber\\
		&\leq\frac{\rho_1}{1-\eta^{-1}\|\Abf_0\|_1\Tp}
		\times\frac{2(1+\eta^{-1}\|\Abf_0\|_2\Tp)\frac{d_{\Fcal}(\Ych) +\eta^{-1}\|\nabla q_0(\Ych)\|_{\Frm}}{1-\eta^{-1}\|\Abf_0\|_2\Tp}
			+\eta^{-1}\|\nabla q_0(\Ych)\|_{\Frm}}{s(\Ych) - 2\rho_2\times \frac{d_{\Fcal}(\Ych) +\eta^{-1}\|\nabla q_0(\Ych)\|_{\Frm}}{1-\eta^{-1}\|\Abf_0\|_2\Tp}}\nonumber\\
		&=\frac{\rho_1}{1-\eta^{-1}\|\Abf_0\|_1\Tp}
		\times \nonumber\\
		&\frac{2(1+\eta^{-1}\|\Abf_0\|_2\Tp)(d_{\Fcal}(\Ych) +\eta^{-1}\|\nabla q_0(\Ych)\|_{\Frm})
			+\eta^{-1}(1-\eta^{-1}\|\Abf_0\|_2\Tp)\|\nabla q_0(\Ych)\|_{\Frm}}{(1-\eta^{-1}\|\Abf_0\|_2\Tp)s(\Ych) - 2\rho_2(d_{\Fcal}(\Ych) +\eta^{-1}\|\nabla q_0(\Ych)\|_{\Frm})}\nonumber\\
		&=\frac{\rho_1}{1-\eta^{-1}\|\Abf_0\|_1\Tp}
		\times \nonumber\\
		&\frac{2d_{\Fcal}(\Ych)
			+\left(2\|\Abf_0\|_2\Tp d_{\Fcal}(\Ych)+3\|\nabla q_0(\Ych)\|_{\Frm}\right)\eta^{-1}
			+\|\Abf_0\|_2\Tp\|\nabla q_0(\Ych)\|_{\Frm}\eta^{-2}}
		{s(\Ych)-2\rho_2 d_{\Fcal}(\Ych) -\big(\|\Abf_0\|_2\Tp s(\Ych) + 2\rho_2\|\nabla q_0(\Ych)\|_{\Frm}\big)\eta^{-1}}\nonumber\\
		&\leq\frac{2\rho_1 d_{\Fcal}(\Ych)
	+2\rho_1\left(\|\Abf_0\|_2\Tp d_{\Fcal}(\Ych)+2\|\nabla q_0(\Ych)\|_{\Frm}\right)\eta^{-1}}
{\big(s(\Ych)-2\rho_2 d_{\Fcal}(\Ych)\big)\big( 1-\eta^{-1}\|\Abf_0\|_1\Tp\big) -\big(\|\Abf_0\|_2\Tp s(\Ych) + 2\rho_2\|\nabla q_0(\Ych)\|_{\Frm}\big)\eta^{-1}}\nonumber\\
&\leq1
	\end{align}
	and, the last inequality is equivalent to \eqref{thm_nfeas_cond1}. Hence:
	\begin{align}
		\eta-\|\Abf_0\|_1\Tp-\rho_1\times \|\taust\|_{2} > 0,
	\end{align}
	which implies that the matrix $\eta\Ibf_n+\Abf_0+p_{\Bcal_{\Yst}}(\taust)$ is diagonally-dominant and according to Lemma \ref{lm:exact}, the point $(\Yst,\Yst\Yst\T)$ is the unique optimal solution for problem \eqref{prob_relax_pen_1a} -- \eqref{prob_relax_pen_1e}.
	
	Additionally, let $\bar{\Ybf}$ be an arbitrary member of $\{\Ybf\in\Fcal\,|\,\|\Ybf-\Ych\|_{\Frm}=d_{\Fcal}(\Ych)\}$. Due to optimality of $\Yst$, 
	we have:
	\begin{subequations}
		\begin{align}
			q_0(\Yst)+ &\eta\times d_{\Fcal}(\Ych)^2 \nonumber\\
			&\leq q_0(\Yst)+\eta\|\Yst-\Ych\|^2_{\Frm}\\
			&\leq q_0(\bar{\Ybf})+\eta\|\bar{\Ybf}-\Ych\|_{\Frm}^2\\
			&\leq q_0(\Ych) + \|\nabla q_0(\Ych)\|_{\Frm}\|\bar{\Ybf}-\Ych\|_{\Frm}
			+\big(\eta+\|\Abf_0\|\Tp_2\big)\|\bar{\Ybf}-\Ych\|_{\Frm}^2\!\!\!\!\\
			&= q_0(\Ych) + \|\nabla q_0(\Ych)\|_{\Frm} d_{\Fcal}(\Ych)+\big(\eta+\|\Abf_0\|\Tp_2\big)d_{\Fcal}(\Ych)^2\\
			&= \tilde{q}_0(\Ych) + \eta\times d_{\Fcal}(\Ych)^2,
		\end{align}
	\end{subequations}
	which concludes $q_0(\Yst)\leq \tilde{q}_0(\Ych)$. 
	\qed\end{proof}

\begin{proof}[Theorem \ref{thm_feas}] The proof is a consequence of Theorem \ref{thm_nfeas} with $d_{\Fcal}(\Ych)=0$.
	\qed\end{proof}

\begin{lemma}\label{lm:contin}
	Define $h_{\eta}:\Fcal\to\Fcal$ as the function mapping any initial point $\Ych$ of problem \eqref{aux_1a} -- \eqref{aux_1c} to its primal solution. Consider an arbitrary $\Ych\in\Fcal$. If the primal solution $\Yst$ is LICQ and the dual point $\taust\in\Rbb^{|\Bcal_{\Yst}|}$ satisfies:
	\begin{align}
		\eta\Ibf_{\! n}+\Abf\Tp_0+p_{\Bst}(\taust)\succ 0,\label{lm:con_asmp}
	\end{align}
	then, $h_{\eta}$ is continuous at point $\Ych$, where $\Bcal_{\Yst}\subseteq\Ecal\cup\Ical$ denotes the set of binding constraints for $\Yst$.
\end{lemma}
\begin{proof}
	Due to LICQ regularity of $\Yst$, the KKT stationarity and primal feasibility conditions are satisfied for $(\Yst,\taust)$:
	\begin{subequations}
		\begin{align}
			\left(\eta\Ibf_n+\Abf\Tp_0+p\Tp_{\Bst}(\taust)\right)\Yst+\Bbf\Tp_0
			+{\sum}_{k\in\Bcal_{\Yst}}\;\accentset{\ast}{\tau}_k\Bbf_k\Tp&=\eta\Ych,\\
			q_k(\Yst)&=0, \qquad\qquad k\in\Bcal_{\Yst}.
		\end{align}
	\end{subequations}
	The Jacobian matrix of the above equations is:
	\begin{align}
		\begin{bmatrix}
			2\!\left(\eta\Ibf_{\! n}+\Abf\Tp_0+p_{\Bst}(\taust)\right) \otimes \Ibf_{\! m} & \Jbf_{\!\Bst}(\Yst)\\
			\Jbf_{\!\Bst}(\Yst)\T & \boldsymbol{0}_{|\Bst|\times|\Bst|}
		\end{bmatrix},
	\end{align}
	which is non-singular due to non-singularity $\Jbf_{\!\Bst}(\Yst)$ and assumption \eqref{lm:con_asmp}.
	\qed\end{proof}

\begin{proof}[Theorem \ref{thm_conv}] 
	Let $\big\{\Ybf^{(l)}\big\}^{\infty}_{l=0}$ denote the sequence generated by Algorithm \ref{alg:1}. Assume by induction that $\Ybf^{(l)}\in\Fch$ and let $\Yst$ be an optimal solution for problem \eqref{aux_1a} -- \eqref{aux_1c} for $\Ych=\Ybf^{(l)}$. According to the optimality of $\Yst$ and feasibility of $\Ybf^{(l)}$:
	\begin{align}
		q_0(\Yst)+\eta\|\Yst-\Ybf^{(l)}\|^2_{\Frm}\leq q_0(\Ybf^{(l)})+\eta\|\Ybf^{(l)}-\Ybf^{(l)}\|^2_{\Frm},\label{decr}
	\end{align}
	which concludes $q_0(\Yst)\leq q_0(\Ybf^{(l)})$, and therefore $\Yst\in\Fch$.
	
	Due to the assumption $\min_{\Ybf\in\Fch}\{s(\Ybf)\}>0$, the point $\Yst$ is LICQ. As a result, there exists a vector of dual variables $\taust\in\Rbb^{|\Bcal_{\Yst}|}$ corresponding to $\Yst$, where $\Bcal_{\Yst}\subseteq\Ecal\cup\Ical$ is the set of binding constraints for $\Yst$. Now, according to Lemmas \ref{lm:d_bound} and \ref{lm:tau_bound}, as well as the assumption \eqref{asmp_conv}, we have:
	\begin{subequations}
		\begin{align}
			&\frac{\rho_1\times \|\eta^{-1}\taust\|_{2}}{1-\eta^{-1}\|\Abf_0\|\Tp_1}
			\leq\frac{\rho_1}{1-\eta^{-1}\|\Abf_0\|\Tp_1}
			\times\frac{2\|\Yst-\Ybf^{(l)}\|_{\Frm}
				+\eta^{-1}\|\nabla q_0(\Yst)\|_{\Frm}}{s(\Yst)}\\
			&\leq\frac{\rho_1}{1-\eta^{-1}\|\Abf_0\|\Tp_1}
			\times\frac{2\times\frac{\eta^{-1}\|\nabla q_0(\Ybf^{(l)})\|_{\Frm}}{1-\eta^{-1}\|\Abf_0\|\Tp_2}
				+\eta^{-1}\|\nabla q_0(\Yst)\|_{\Frm}}{s(\Yst)}\\
			&\leq\frac{\rho_1\,\eta^{-1}\big(2\|\nabla q_0(\Ybf^{(l)})\|_{\Frm}+(1-\eta^{-1}\|\Abf_0\|\Tp_2)\|\nabla q_0(\Yst)\|_{\Frm}\big)}{\big(1-\eta^{-1}\|\Abf_0\|\Tp_1\big)\big(1-\eta^{-1}\|\Abf_0\|\Tp_2\big)s(\Yst)}\\
			&\leq\frac{\rho_1\,\eta^{-1}\big(2\|\nabla q_0(\Ybf^{(l)})\|_{\Frm}+(1-\eta^{-1}\|\Abf_0\|\Tp_2)\|\nabla q_0(\Yst)\|_{\Frm}\big)}{\big(1-\eta^{-1}\|\Abf_0\|\Tp_1-\eta^{-1}\|\Abf_0\|\Tp_2\big)s(\Yst)}\\
			&\leq\frac{\eta^{-1}\rho_1}{1-\eta^{-1}\|\Abf_0\|\Tp_1-\eta^{-1}\|\Abf_0\|\Tp_2}
			\times\frac{2\|\nabla q_0(\Ybf^{(l)})\|_{\Frm}+\|\nabla q_0(\Yst)\|_{\Frm}}{s(\Yst)}\\
			&\leq\frac{3\rho_1}{\eta-\|\Abf_0\|\Tp_1-\|\Abf_0\|\Tp_2}
			\times\frac{\max_{\Ybf\in\Fch}\{\|\nabla q_0(\Ybf)\|_{\Frm}\}}
			{\min_{\Ybf\in\Fch}\{s(\Ybf)\}}\leq 1.
		\end{align}
	\end{subequations}
	Hence:
	\begin{align}
		\eta-\|\Abf_0\|_1\Tp-\rho_1\times \|\taust\|_{2} > 0,
	\end{align}
	which implies that the matrix $\eta\Ibf_n+\Abf_0+p_{\Bcal_{\Yst}}(\taust)$ is diagonally-dominant and according to Lemma \ref{lm:exact}, the point $(\Yst,\Yst\Yst\T)$ is the unique optimal solution for problem \eqref{prob_relax_pen_1a} -- \eqref{prob_relax_pen_1e}, which implies that $\Ybf^{(l+1)}=\Yst$.
	
	On the other hand, according to \eqref{decr}, the sequence $\big\{q_0(\Ybf^{(l)})\big\}^{\infty}_{l=0}$ is non-increasing and convergent. Additionally, we have:
	\begin{align}
		\|\Ybf^{(l+1)}-\Ybf^{(l)}\|^2_{\Frm}\leq \eta^{-1}\left(q_0(\Ybf^{(l)}) - q_0(\Ybf^{(l+1)})\right),
	\end{align}
	which means that $\big\{\Ybf^{(l)}\big\}^{\infty}_{l=0}$ is convergent to a point $\Ybf^{(\infty)}\in\Fch$ due to compactness of $\Fch$.
	
	According to Lemma \eqref{lm:contin}, we have $h_{\eta}\big(\Ybf^{(\infty)}\big)=\Ybf^{(\infty)}$ (See Lemma \ref{lm:contin} for the definition of $h_{\eta}$). Now, due to LICQ regularity of $\Ybf^{(\infty)}$, there exists a vector of dual variables $\taubf^{(\infty)}\in\Rbb^{|\Bcal^{(\infty)}|}$ for which the following KKT stationarity and primal feasibility conditions are satisfied:
	\begin{subequations}
		\begin{align}
			\left(\Abf\Tp_0+p\Tp_{\Bst}\big(\taubf^{(\infty)}\big)\right)\Ybf^{(\infty)}+\Bbf\Tp_0
			+{\sum}_{k\in\Bcal_{\Ybf^{(\infty)}}}\;\tau^{(\infty)}_k\Bbf_k\Tp&=\mathbf{0}_{n\times m},\label{kkt_local1}\\
			q_k\big(\Ybf^{(\infty)}\big)&=0, \quad k\in\Bcal_{\Ybf^{(\infty)}},\label{kkt_local2}
		\end{align}
	\end{subequations}
	where $\Bcal^{(\infty)}$ denotes the set of binding constraints for $\Ybf^{(\infty)}$. Finally, it can be easily observed that \eqref{kkt_local1} -- \eqref{kkt_local2} are equivalent to KKT optimality conditions for the original non-convex problem \eqref{QCQP_tensor_a} -- \eqref{QCQP_tensor_c} which completes the proof.
	\qed\end{proof}

\section{Computational Results}

In this section, three case studies are given to demonstrate some possible applications and to evaluate the merits of the proposed parabolic relaxation. In Section \ref{sec_QPLIB}, we compute lower bounds offered by the parabolic relaxation for benchmark cases from the library of quadratic programming instances (QPLIB) \cite{FuriniEtAl2017TR}, and compare them to that of SOCP relaxation. In Section \ref{sec_penalty}, we evaluate the ability of penalized parabolic relaxation in finding feasible points for non-convex QCQPs with continuous variables. In Section \ref{sec_sysid}, we discuss the problem of identifying linear dynamical systems based on limited snapshots from sample trajectories.
All of the experiments are performed on a desktop computer with a 12-core 3.0GHz CPU and 256GB RAM. MOSEK v8.1 \cite{mosek} is used through MATLAB 2017a to solve the resulting convex relaxations.

\subsection{Bounds for Non-convex QCQP}\label{sec_QPLIB}

\begin{table}[t!]
	\centering
	\caption{
		Comparison of lower bounds from the SOCP and the parabolic relaxations for QPLIB instances.}
	\scalebox{0.75}{
		\begin{tabular}{ c|c c c c|c c c|c c c }
			\hline\hline
			\multirow{ 2}{*}{Inst}  & Total & Quad & Total & Optimal &
			\multicolumn{3}{c|}{SOCP} & \multicolumn{3}{c}{Parabolic} \\
			\cline{6-11}
			& Var & Cons & Cons &  Cost 
			& LB & GAP(\%) & $t(s)$ 
			& LB & GAP(\%)& $t(s)$\\
			\hline\hline
			0018 & 50 & 0 & 1 & -6.386 & - & - & - & - & - & - \\
			0343 & 50 & 0 & 1 & -6.386 & \bf -226.334 & \bf 3444.22 & 1.56 & -234.706 & 3575.32 & 0.63 \\
			0911 & 50 & 50 & 50 & -32.148 & -299.520 & 831.71 & 0.63 & \bf -76.525 & \bf 138.04 & 0.58 \\
			0975 & 50 & 10 & 10 & -37.854 & -295.158 & 679.74 & 0.63 & \bf -78.384 & \bf 107.07 & 0.56 \\
			1055 & 40 & 20 & 20 & -33.037 & -393.235 & 1090.29 & 0.55 & \bf -94.630 & \bf 186.44 & 0.59 \\
			1143 & 40 & 20 & 24 & -57.247 & -600.990 & 949.82 & 0.61 & \bf -179.463 & \bf 213.49 & 0.63 \\
			1157 & 40 & 1 & 9 & -10.948 & -37.239 & 240.14 & 0.72 & \bf -20.431 & \bf 86.61 & 0.61 \\
			1353 & 50 & 1 & 6 & -7.714 & -73.476 & 852.49 & 0.64 & \bf -23.557 & \bf 205.37 & 0.63 \\
			1423 & 40 & 20 & 24 & -14.967 & -92.610 & 518.74 & 0.56 & \bf -31.901 & \bf 113.14 & 0.58 \\
			1437 & 50 & 1 & 11 & -7.789 & -67.007 & 760.25 & 0.75 & \bf -28.147 & \bf 261.36 & 0.64 \\
			1451 & 60 & 60 & 66 & -87.576 & -793.469 & 806.03 & 0.70 & \bf -227.164 & \bf 159.39 & 0.70 \\
			1493 & 40 & 1 & 5 & -43.160 & -407.538 & 844.24 & 0.58 & \bf -142.863 & \bf 231.01 & 0.56 \\
			1507 & 30 & 30 & 33 & -8.301 & -55.646 & 570.32 & 0.59 & \bf -16.726 & \bf 101.49 & 0.56 \\
			1535 & 60 & 60 & 66 & -11.586 & -122.272 & 955.33 & 0.61 & \bf -40.866 & \bf 252.71 & 0.66 \\
			1619 & 50 & 25 & 30 & -9.217 & -135.057 & 1365.26 & 0.67 & \bf -31.556 & \bf 242.35 & 0.55 \\
			1661 & 60 & 1 & 13 & -15.955 & -114.224 & 615.92 & 0.70 & \bf -46.364 & \bf 190.60 & 0.66 \\
			1675 & 60 & 1 & 13 & -75.669 & -670.112 & 785.59 & 0.66 & \bf -198.723 & \bf 162.62 & 0.63 \\
			1703 & 60 & 30 & 36 & -132.802 & -1197.274 & 801.55 & 0.66 & \bf -411.862 & \bf 210.13 & 0.67 \\
			1745 & 50 & 50 & 55 & -72.377 & -378.953 & 423.58 & 0.61 & \bf -138.833 & \bf 91.82 & 0.61 \\
			1773 & 60 & 1 & 7 & -14.642 & -169.958 & 1060.77 & 0.64 & \bf -49.329 & \bf 236.90 & 0.69 \\
			1886 & 50 & 50 & 50 & -78.672 & -627.305 & 697.37 & 0.72 & \bf -163.551 & \bf 107.89 & 0.67 \\
			1913 & 48 & 48 & 48 & -52.108 & -289.660 & 455.88 & 0.67 & \bf -82.897 & \bf 59.09 & 0.63 \\
			1922 & 30 & 60 & 60 & -35.951 & -216.084 & 501.06 & 0.63 & \bf -62.914 & \bf 75.00 & 0.69 \\
			1931 & 40 & 40 & 40 & -55.709 & -390.254 & 600.52 & 0.66 & \bf -103.182 & \bf 85.22 & 0.59 \\
			1940 & 48 & 96 & 96 & -38.310 & -283.950 & 641.19 & 0.70 & \bf -69.374 & \bf 81.08 & 0.61 \\
			1967 & 50 & 75 & 75 & -107.581 & -1214.746 & 1029.14 & 0.73 & \bf -306.963 & \bf 185.33 & 0.67 \\
			\hline\hline
	\end{tabular}}
	\label{tab:cases}
\end{table}

This case study is concerned with assessing the tightness of parabolic relaxation in comparison with SOCP relaxation. To this end, we perform experiments on the non-convex QCQP problems from QPLIB. The size, number of constraints, and optimal cost for each instance are reported in Table \ref{tab:cases}. The following valid inequalities are imposed on all convex relaxations:
\begin{subequations}
	\begin{align}
		&X_{kk}-(x^{\mathrm{lb}\!\!\!\!\!\phantom{\mathrm{ub}}}_k+x^{\mathrm{ub}\!\!\!\!\!\phantom{\mathrm{lb}}}_k)x_k 
		+x^{\mathrm{lb}\!\!\!\!\!\phantom{\mathrm{ub}}}_k x^{\mathrm{ub}\!\!\!\!\!\phantom{\mathrm{lb}}}_k \leq 0,&\qquad\forall k\in\Ncal,\label{lbub}\\
		&X_{kk}-(x^{\mathrm{ub}\!\!\!\!\!\phantom{\mathrm{lb}}}_k+x^{\mathrm{ub}\!\!\!\!\!\phantom{\mathrm{lb}}}_k)x_k 
		+x^{\mathrm{ub}\!\!\!\!\!\phantom{\mathrm{lb}}}_k x^{\mathrm{ub}\!\!\!\!\!\phantom{\mathrm{lb}}}_k \geq 0,&\qquad\forall k\in\Ncal,\label{ubub}\\
		&X_{kk}-(x^{\mathrm{lb}\!\!\!\!\!\phantom{\mathrm{ub}}}_k+x^{\mathrm{lb}\!\!\!\!\!\phantom{\mathrm{ub}}}_k)x_k 
		+x^{\mathrm{lb}\!\!\!\!\!\phantom{\mathrm{ub}}}_k x^{\mathrm{lb}\!\!\!\!\!\phantom{\mathrm{ub}}}_k \geq 0,&\qquad\forall k\in\Ncal,\label{lblb}
	\end{align}
\end{subequations}
where $\boldsymbol{x}^{\mathrm{lb}},\boldsymbol{x}^{\mathrm{ub}}\in\mathbb{R}^n$ are given lower and upper bounds on $\boldsymbol{x}$ and $\Ncal\triangleq\{1,\ldots,n\}$. Lower bounds from parabolic and SOCP relaxations and their gaps from the optimality are reported in Table \ref{tab:cases} where:
\begin{align}
	\mathrm{GAP}(\%)&\triangleq 100\times\frac{\qbf_0(\boldsymbol{x}^{\mathrm{QPLIB}})-\qbf_{0}(\xst,\Xst)}{|\qbf_0(\boldsymbol{x}^{\mathrm{QPLIB}})|},\!\!\!\label{gapgap}
\end{align}
and $\xbf^{\mathrm{QPLIB}}$ is the optimal point provided by QPLIB, and $(\xst,\Xst)$ represent the outcome of convex relaxation. As shown by the table, for all of the cases except 0018 and 0343, the parabolic relaxation outperforms SOCP by an average of $572\%$. For the case 0018, both relaxations are unbounded.

\subsection{Feasible Point Recovery}\label{sec_penalty}

\begin{table}[t!]
	\centering
	\caption{
		Comparison between feasible points offered by penalized parabolic relaxations and GUROBI 9.0.}
	\scalebox{0.75}{
		\begin{tabular}{ c c | c c c c c c|c c }
			\hline\hline
			\multirow{ 2}{*}{Inst}  & Optimal & 
			\multicolumn{6}{c|}{Penalized Parabolic} & \multicolumn{2}{c}{GUROBI} \\
			\cline{3-10}
			& Cost & $\eta$ & ${i^{\mathrm{feas}}}^{\phantom{1}}$ &  $i^{\mathrm{stop}}$  & $t^{\mathrm{stop}}(s)$ & UB & GAP(\%)
			& UB & GAP(\%) \\
			\hline\hline
			0018 & -6.386 & 5e+2 & 1 & 286 & 120.19 & \bf -6.377 & \bf 0.13 &   -5.939 &   7.00 \\
			0343 & -6.386 & 5e+2 & 1 & 322 & 129.65 & \bf -6.031 & \bf 5.56 &   -5.939 &   7.00 \\
			0911 & -32.148 & 1e+1 & 4 & 63 & 27.99 & \bf -31.713 & \bf 1.35 &   0.000 &   100.00 \\
			0975 & -37.854 & 1e+1 & 1 & 31 & 13.04 & \bf -36.433 & \bf 3.75 &   -19.643 &   48.11 \\
			1055 & -33.037 & 2e+1 & 1 & 65 & 6.65 & \bf -32.775 & \bf 0.79 &   -6.496 &   80.34 \\
			1143 & -57.247 & 5e+1 & 1 & 143 & 57.23 & \bf -55.511 & \bf 3.03 &   -46.370 &   19.00 \\
			1157 & -10.948 & 5e+0 & 1 & 25 & 3.01 & \bf -10.942 & \bf 0.06 &   Inf &   Inf \\
			1353 & -7.714 & 5e+0 & 1 & 52 & 22.61 & \bf -7.708 & \bf 0.08 &   -7.357 &   4.63 \\
			1423 & -14.967 & 5e+0 & 1 & 26 & 10.49 & \bf -14.687 & \bf 1.87 &   Inf &   Inf \\
			1437 & -7.789 & 5e+0 & 1 & 37 & 16.38 & \bf -7.787 & \bf 0.03 &   Inf &   Inf \\
			1451 & -87.576 & 5e+1 & 2 & 142 & 62.92 & \bf -87.357 & \bf 0.25 &   Inf &   Inf \\
			1493 & -43.160 & 5e+1 & 1 & 58 & 22.57 & \bf -41.831 & \bf 3.08 &   -40.379 &   6.44 \\
			1507 & -8.301 & 5e+0 & 1 & 41 & 17.50 & \bf -8.292 & \bf 0.11 &   -5.104 &   38.52 \\
			1535 & -11.586 & 5e+0 & 2 & 123 & 68.76 & \bf -11.521 & \bf 0.56 &   Inf &   Inf \\
			1619 & -9.217 & 5e+0 & 1 & 42 & 17.36 & \bf -9.212 & \bf 0.06 &   Inf &   Inf \\
			1661 & -15.955 & 5e+0 & 1 & 38 & 16.47 & \bf -15.664 & \bf 1.83 &   Inf &   Inf \\
			1675 & -75.669 & 2e+1 & 1 & 23 & 9.65 & \bf -75.539 & \bf 0.17 &   Inf &   Inf \\
			1703 & -132.802 & 5e+1 & 2 & 54 & 22.23 & \bf -132.474 & \bf 0.25 &   Inf &   Inf \\
			1745 & -72.377 & 2e+1 & 1 & 33 & 14.03 & \bf -71.805 & \bf 0.79 &   Inf &   Inf \\
			1773 & -14.642 & 5e+0 & 1 & 72 & 29.75 & \bf -14.177 & \bf 3.18 &   Inf &   Inf \\
			1886 & -78.672 & 2e+1 & 1 & 33 & 13.27 & \bf -78.601 & \bf 0.09 &   -77.185 &   1.89 \\
			1913 & -52.108 & 1e+1 & 9 & 27 & 11.86 & \bf -51.888 & \bf 0.42 &   -50.847 &   2.42 \\
			1922 & -35.951 & 1e+1 & 2 & 25 & 10.21 &   -35.448 &   1.40 & \bf -35.561 & \bf 1.08 \\
			1931 & -55.709 & 2e+1 & 1 & 66 & 27.54 &   -54.290 &   2.55 & \bf -55.345 & \bf 0.65 \\
			1940 & -38.310 & 2e+1 & 1 & 55 & 24.08 & \bf -38.264 & \bf 0.12 &   -37.426 &   2.31 \\
			1967 & -107.581 & 5e+1 & 2 & 75 & 32.86 & \bf -106.861 & \bf 0.67 &   0.000 &   100.00 \\
			\hline\hline
	\end{tabular}}
	\label{tab:cases_penalty}
\end{table}

This case study is concerned with the recovery of feasible points from inexact parabolic relaxations using Algorithm \ref{alg:1}. Table \ref{tab:cases_penalty} demonstrates a comparison between this approach and GUROBI 9.0. Let $(\accentset{\ast}{\boldsymbol{x}},\accentset{\ast}{\boldsymbol{X}})$ denote the optimal solution of the convex relaxation \eqref{prob_relax_pen_1a} -- \eqref{prob_relax_pen_1e}. We use the point $\check{\boldsymbol{x}} = \accentset{\ast}{\boldsymbol{x}}$ as the initial point of the algorithm.

The penalty parameter $\eta$ is chosen via bisection as the smallest number of the form $\alpha\times10^\beta$, which results in a $\accentset{\ast}{\boldsymbol{X}}=\accentset{\ast}{\boldsymbol{x}}\accentset{\ast}{\boldsymbol{x}}^{\top}$ during the first ten rounds, where $\alpha\in\{1,2,5\}$ and $\beta$ is an integer. In all of the experiments, the value of $\eta$ remained static throughout Algorithm \ref{alg:1}. Denote the sequence of penalized SDP solutions obtained by Algorithm \ref{alg:1} as:
\begin{align}
	(\boldsymbol{x}^{(1)},\boldsymbol{X}^{(1)}), \;\;
	(\boldsymbol{x}^{(2)},\boldsymbol{X}^{(2)}), \;\;
	(\boldsymbol{x}^{(3)},\boldsymbol{X}^{(3)}), \;\;\ldots.\nonumber
\end{align}
The smallest $i$ such that:
\begin{align}
	\mathrm{tr}\{\boldsymbol{X}^{(i)}-\boldsymbol{x}^{(i)}(\boldsymbol{x}^{(i)})^{\top}\}<10^{-7},
\end{align}
is denoted by $i^{\mathrm{feas}}$, i.e., it is the number of rounds that Algorithm \ref{alg:1} needs to attain a tight penalization. 
Moreover, the smallest $i$ such that:
\begin{align}
	\frac{q_0(\boldsymbol{x}^{(i-1)})-q_0(\boldsymbol{x}^{(i)})}{|q_0(\boldsymbol{x}^{(i)})|}\leq 10^{-4},\label{scri}
\end{align}
is denoted by $i^{\mathrm{stop}}$, and $\mathrm{UB}\triangleq q_0(\boldsymbol{x}^{(i^{\mathrm{stop}})})$. The following formula is used to calculate the final percentage gaps from the optimal costs reported by the QPLIB library:
\begin{align}
	\mathrm{GAP}(\%)&=100\times\frac{q_{0}^{\mathrm{stop}}-
		q_0(\boldsymbol{x}^{\mathrm{QPLIB}})}{|q_0(\boldsymbol{x}^{\mathrm{QPLIB}})|}.\!\!\!\label{gapgap2}
\end{align}
{Moreover, $t$(s) denotes the cumulative solver time in seconds for the $i^{\mathrm{stop}}$ rounds. Our results are compared with GUROBI 9.0 by fixing the maximum solver time equal to the cumulative solver time spent by Algorithm \ref{alg:1}.} The resulting lower bounds, upper bounds and GAPs (from equation \eqref{gapgap2}) are reported in Table \ref{tab:cases_penalty}.

As demonstrated by Figures \ref{fig_good1} and \ref{fig_good2}, Algorithm \ref{alg:1} outperforms GUROBI for the majority of cases. However, for cases 1922, 1931, 0975, and 1423, GUROBI ultimately outperforms Algorithm \ref{alg:1} and this is shown by Figure \ref{fig_bad}.

\begin{figure*}[h!]
	\centering
	\includegraphics[height=0.24\textwidth]{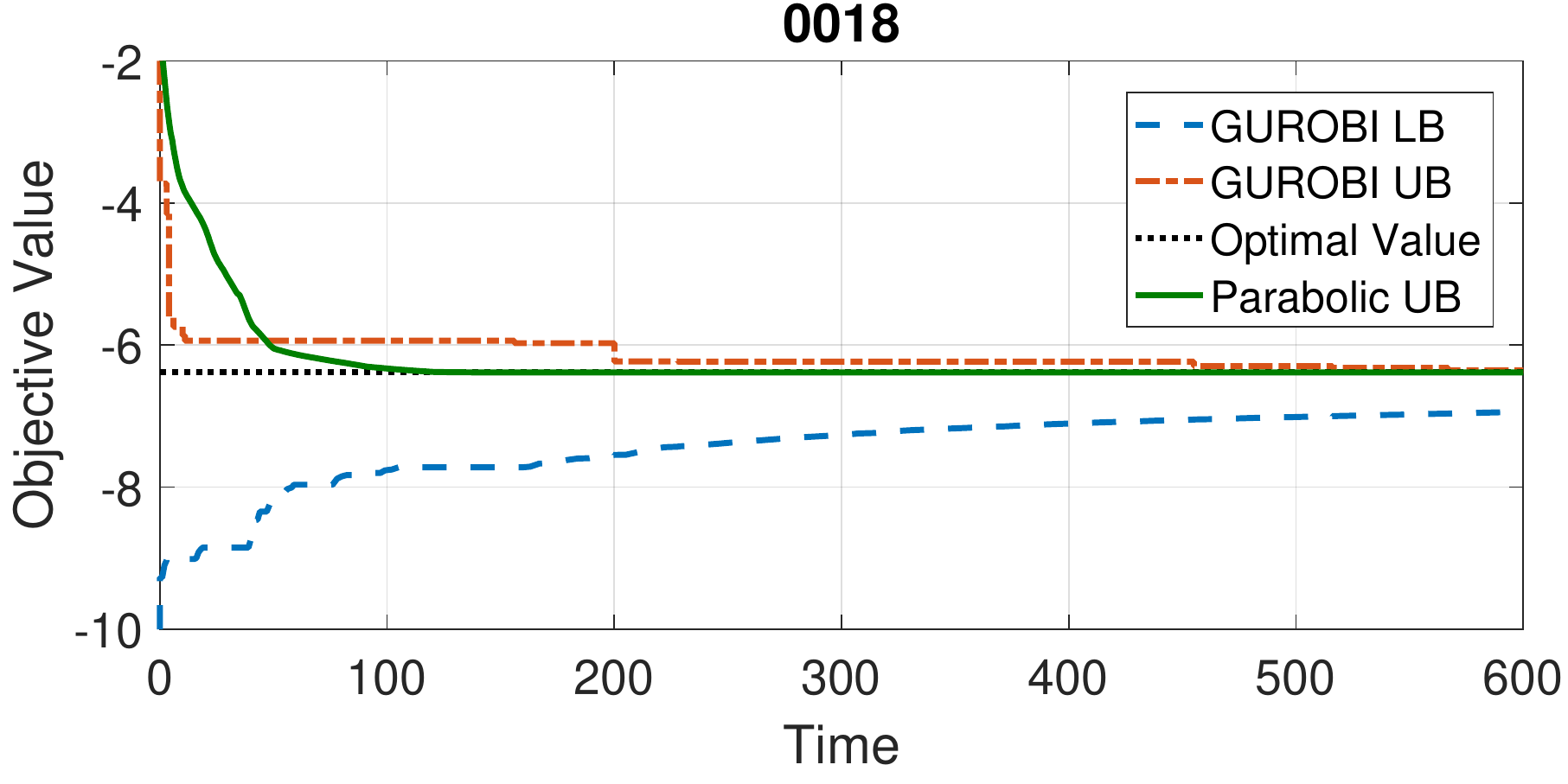}
	\includegraphics[height=0.24\textwidth]{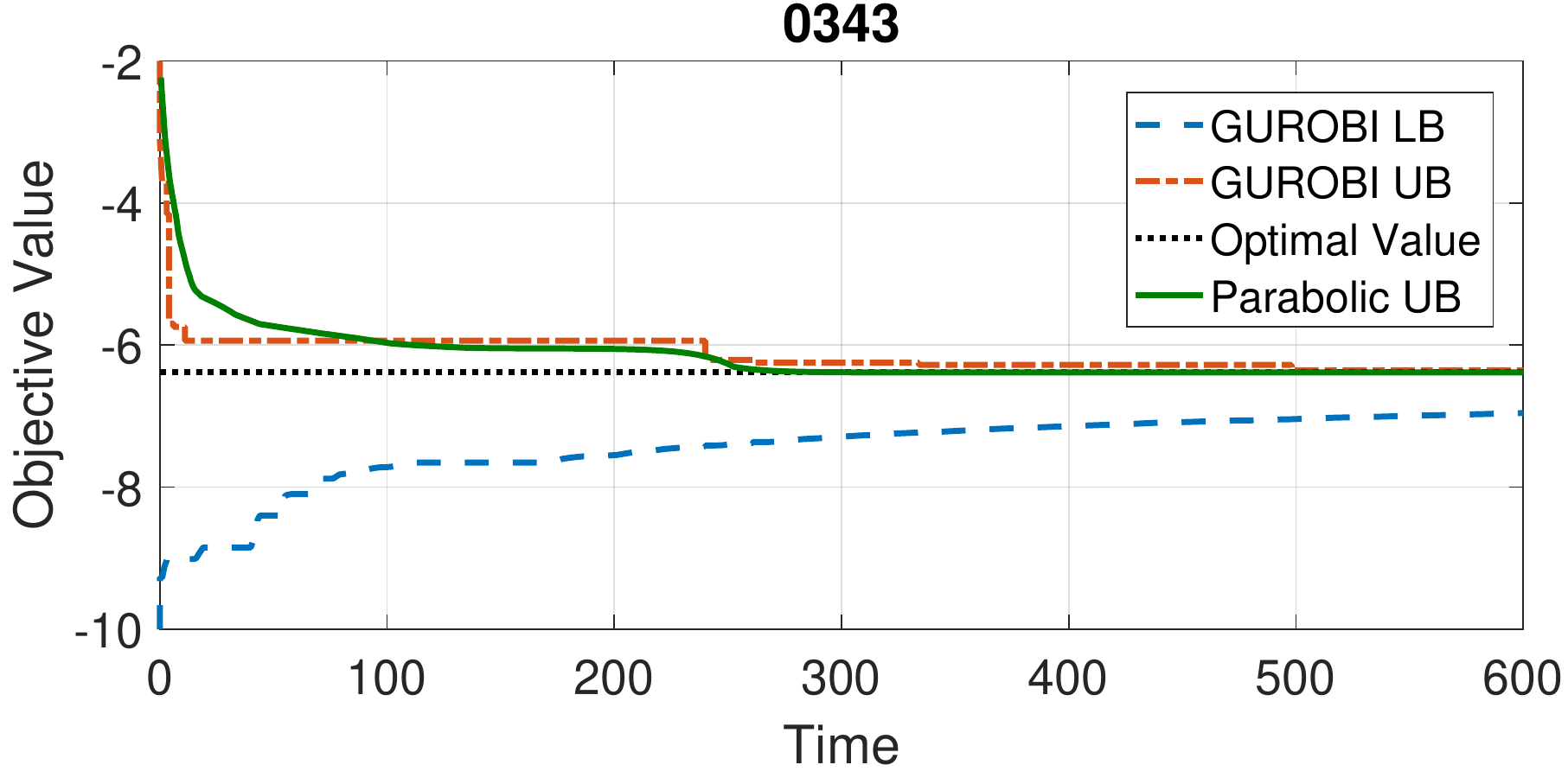}\\
	\vspace{1mm}
	\includegraphics[height=0.24\textwidth]{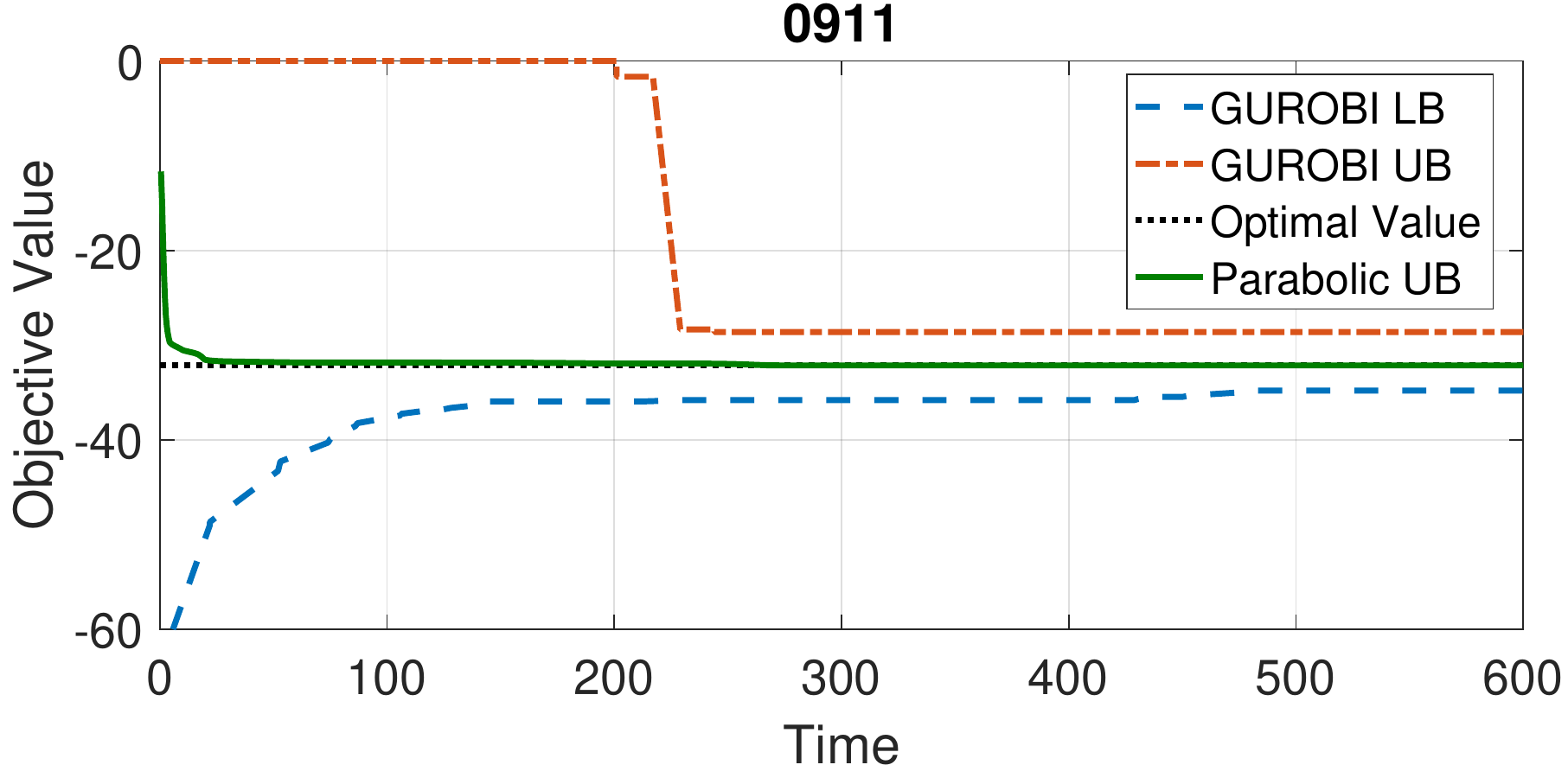}
	\includegraphics[height=0.24\textwidth]{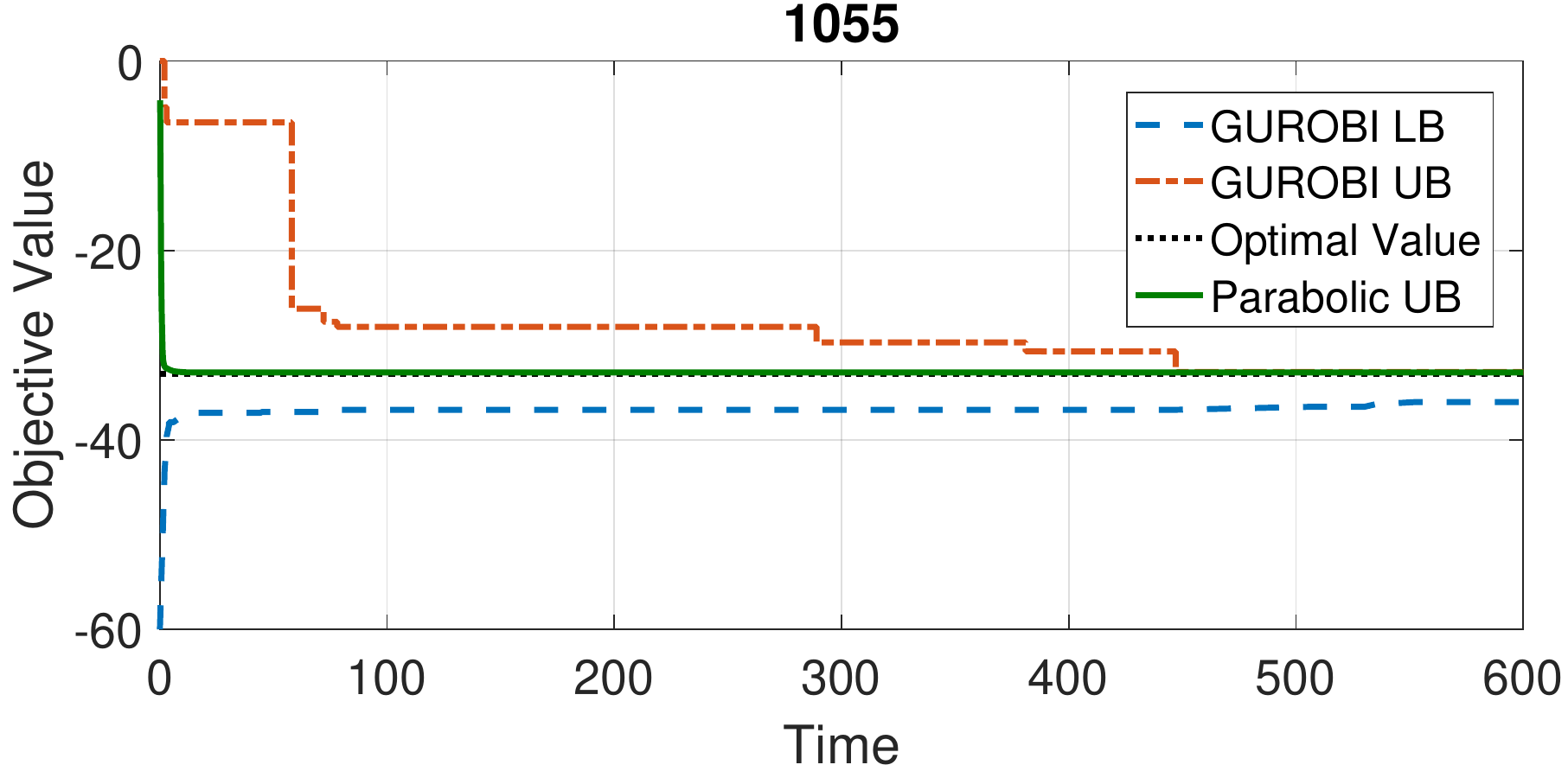}\\
	\vspace{1mm}
	\includegraphics[height=0.24\textwidth]{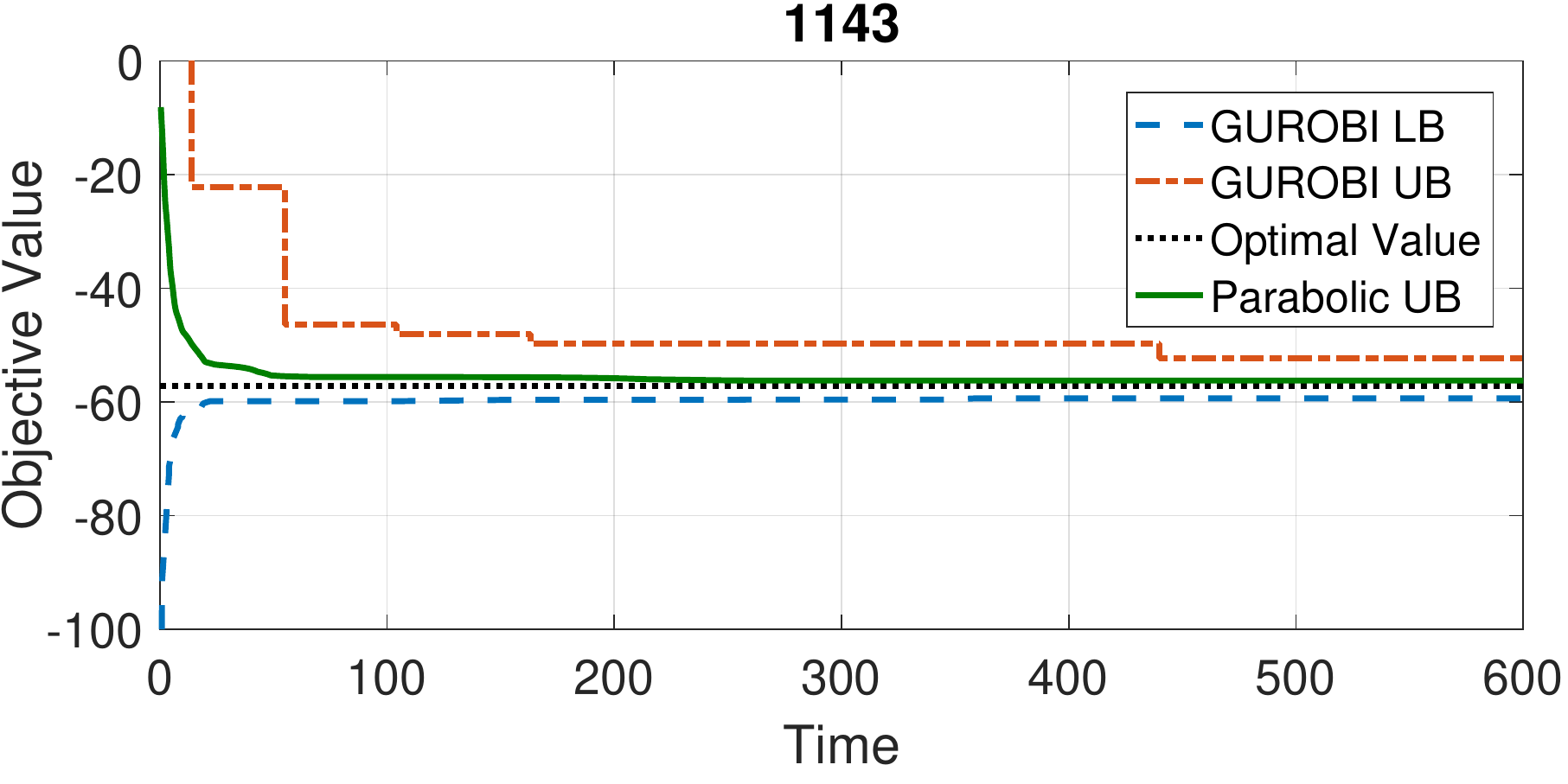}
	\includegraphics[height=0.24\textwidth]{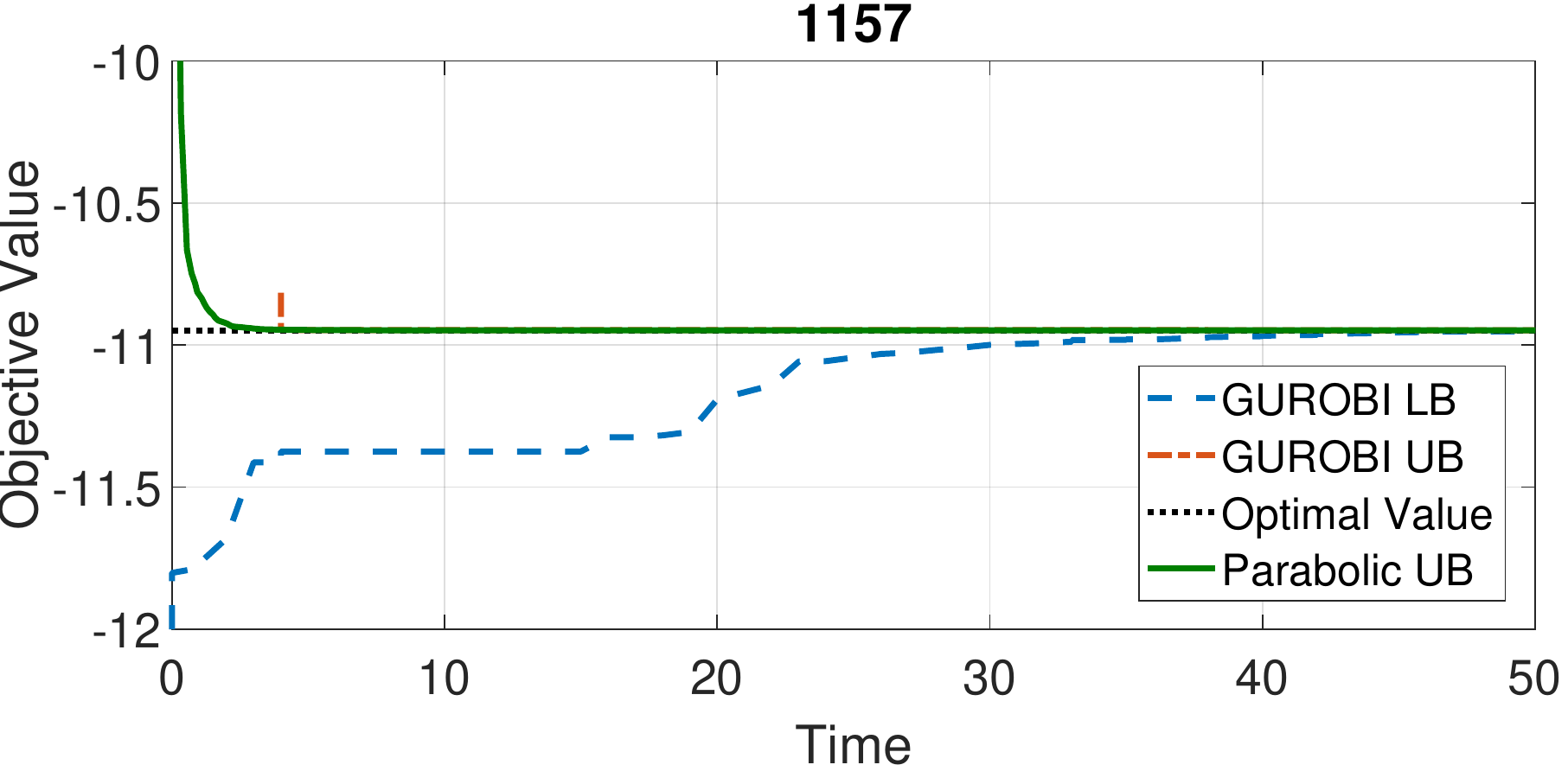}\\
	\vspace{1mm}
	\includegraphics[height=0.24\textwidth]{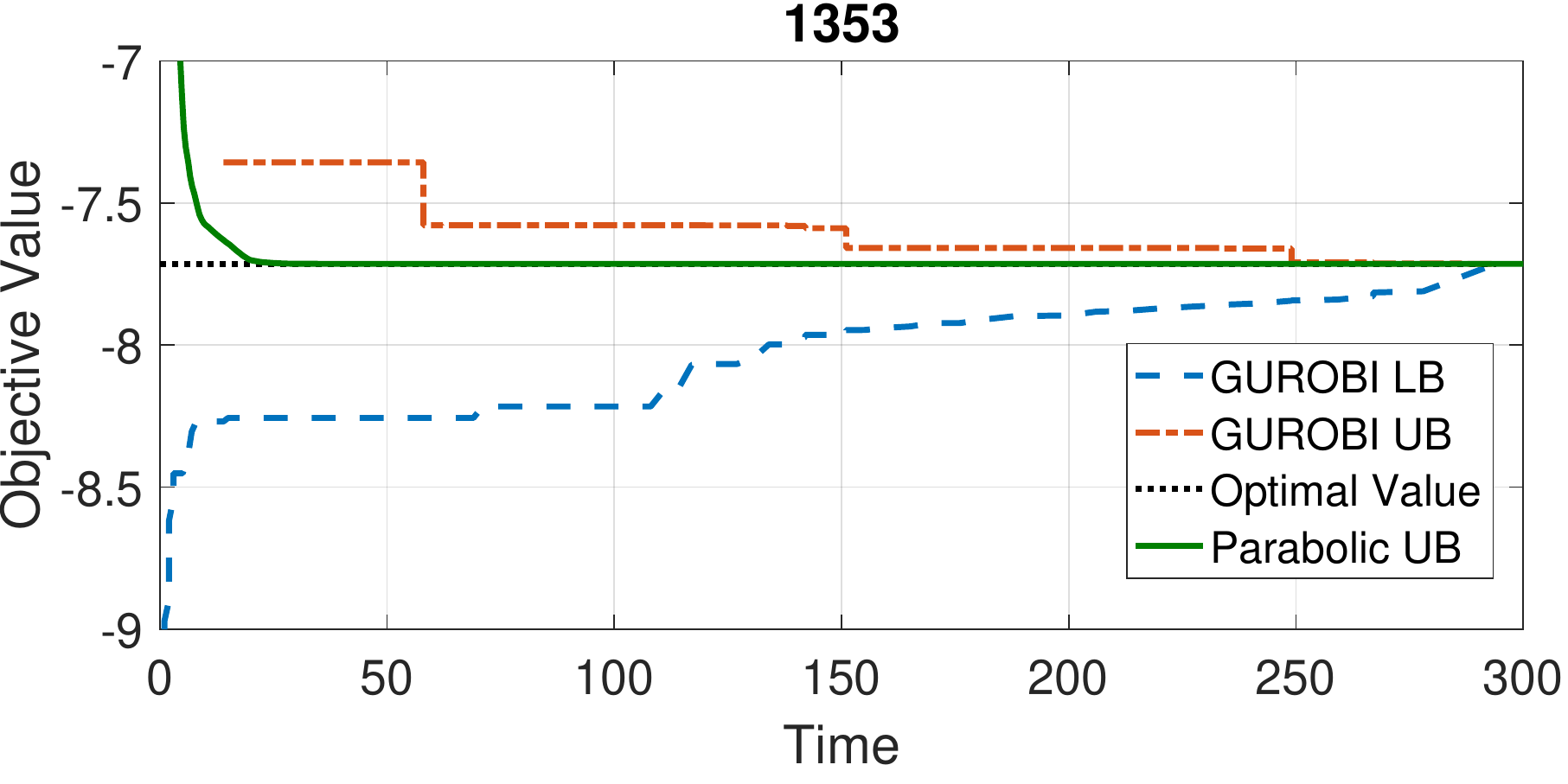}
	\includegraphics[height=0.24\textwidth]{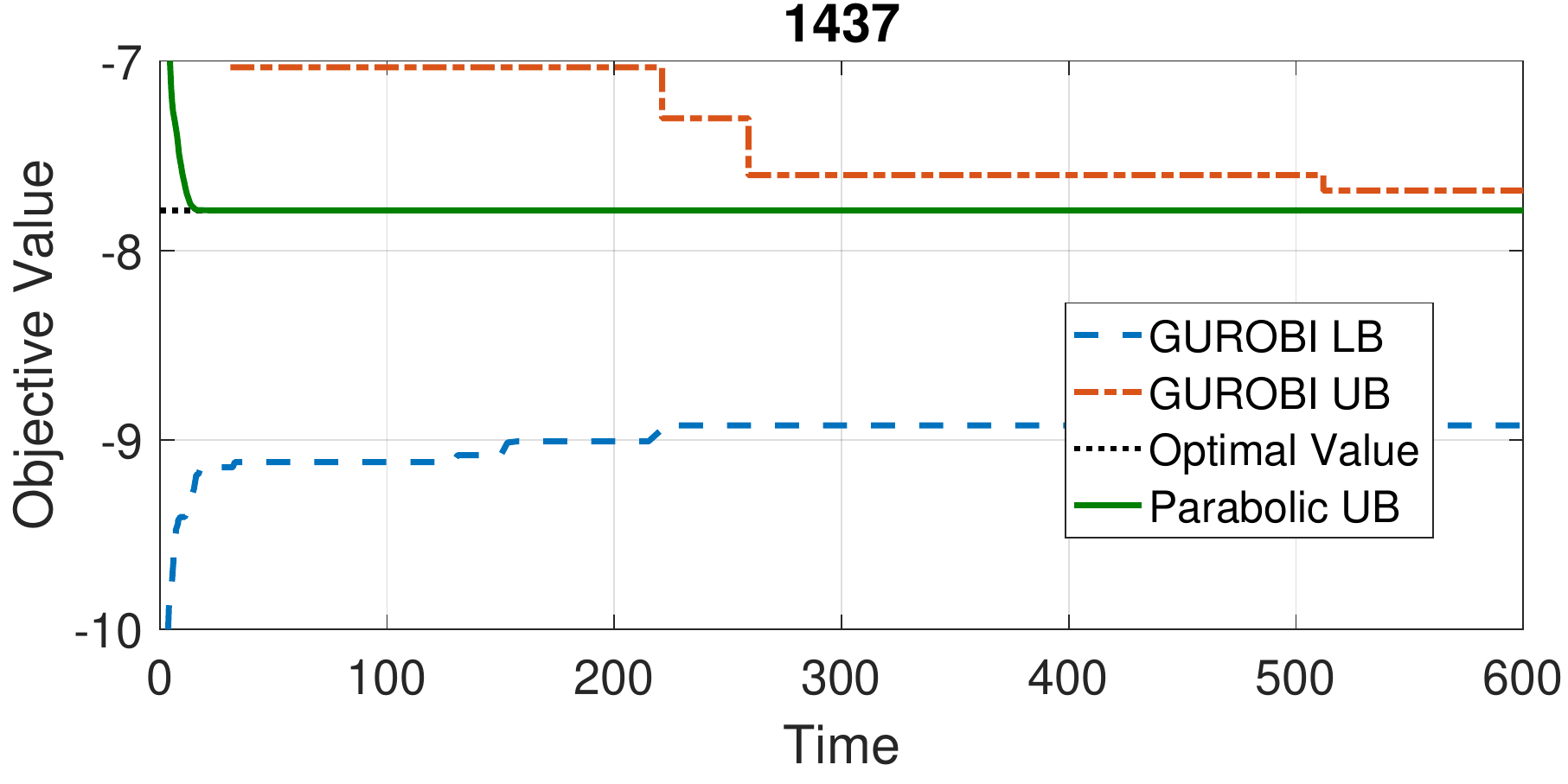}\\
	\vspace{1mm}
	\includegraphics[height=0.24\textwidth]{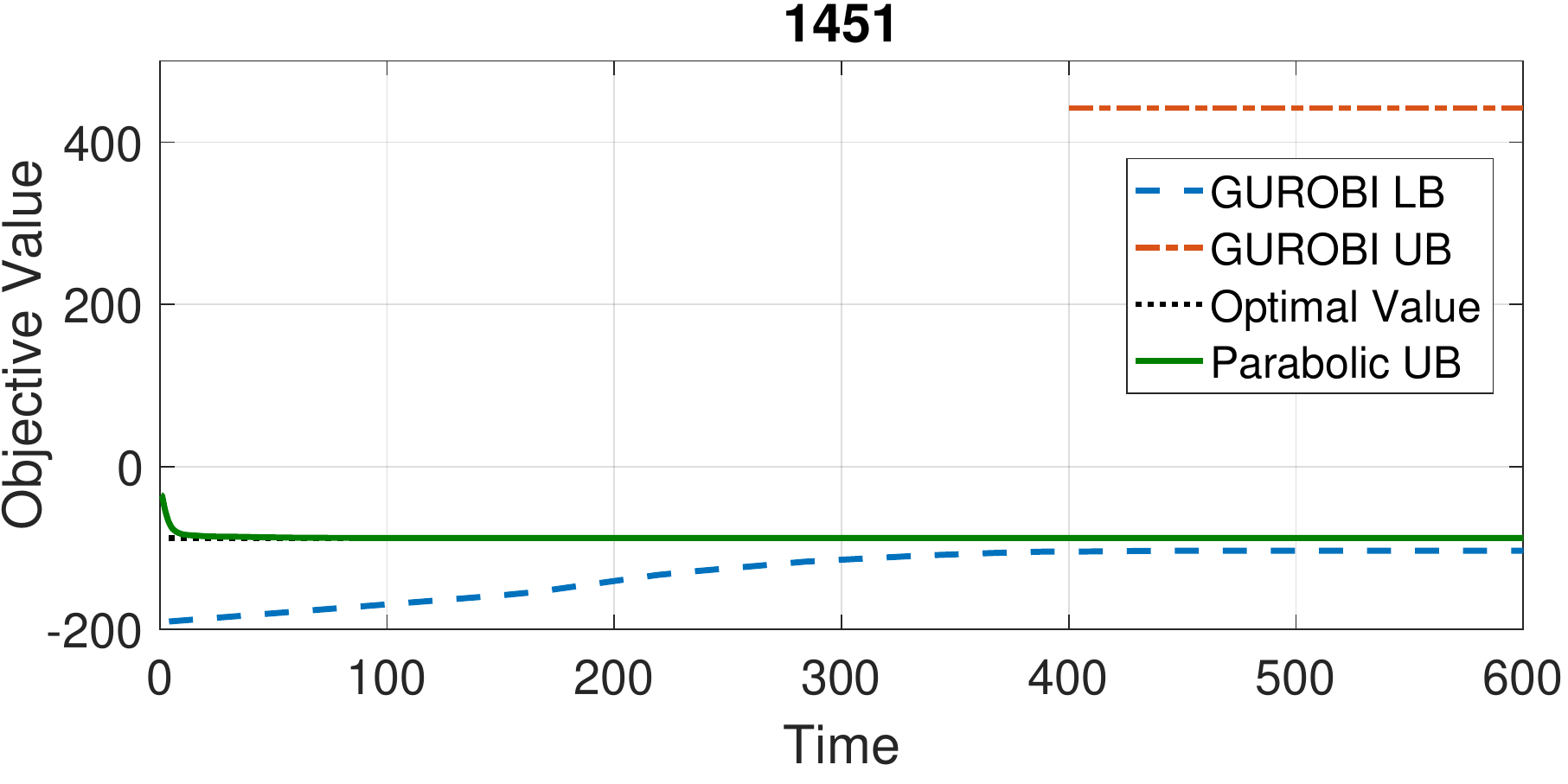}
	\includegraphics[height=0.24\textwidth]{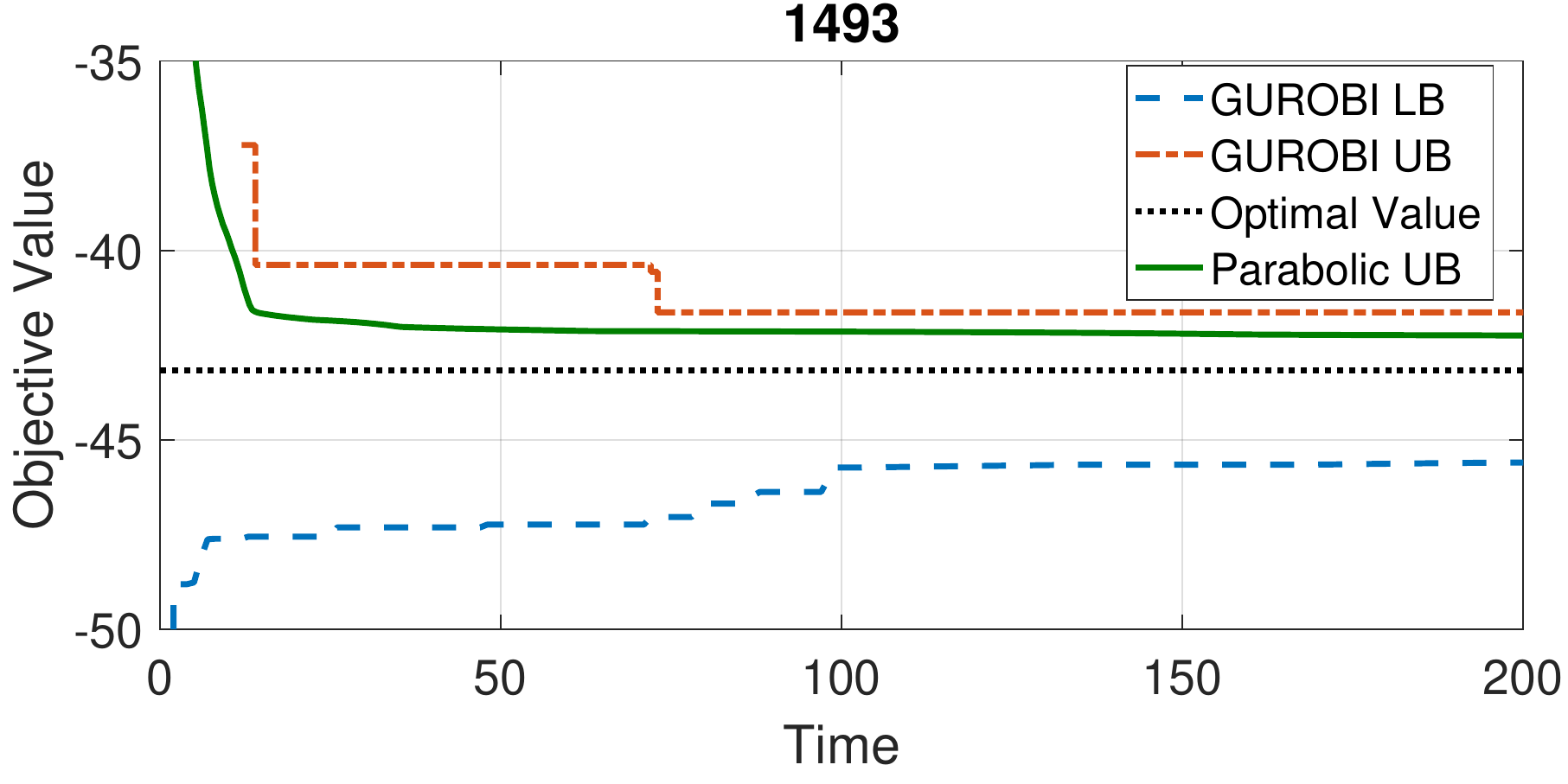}\\
	\vspace{1mm}
	\includegraphics[height=0.24\textwidth]{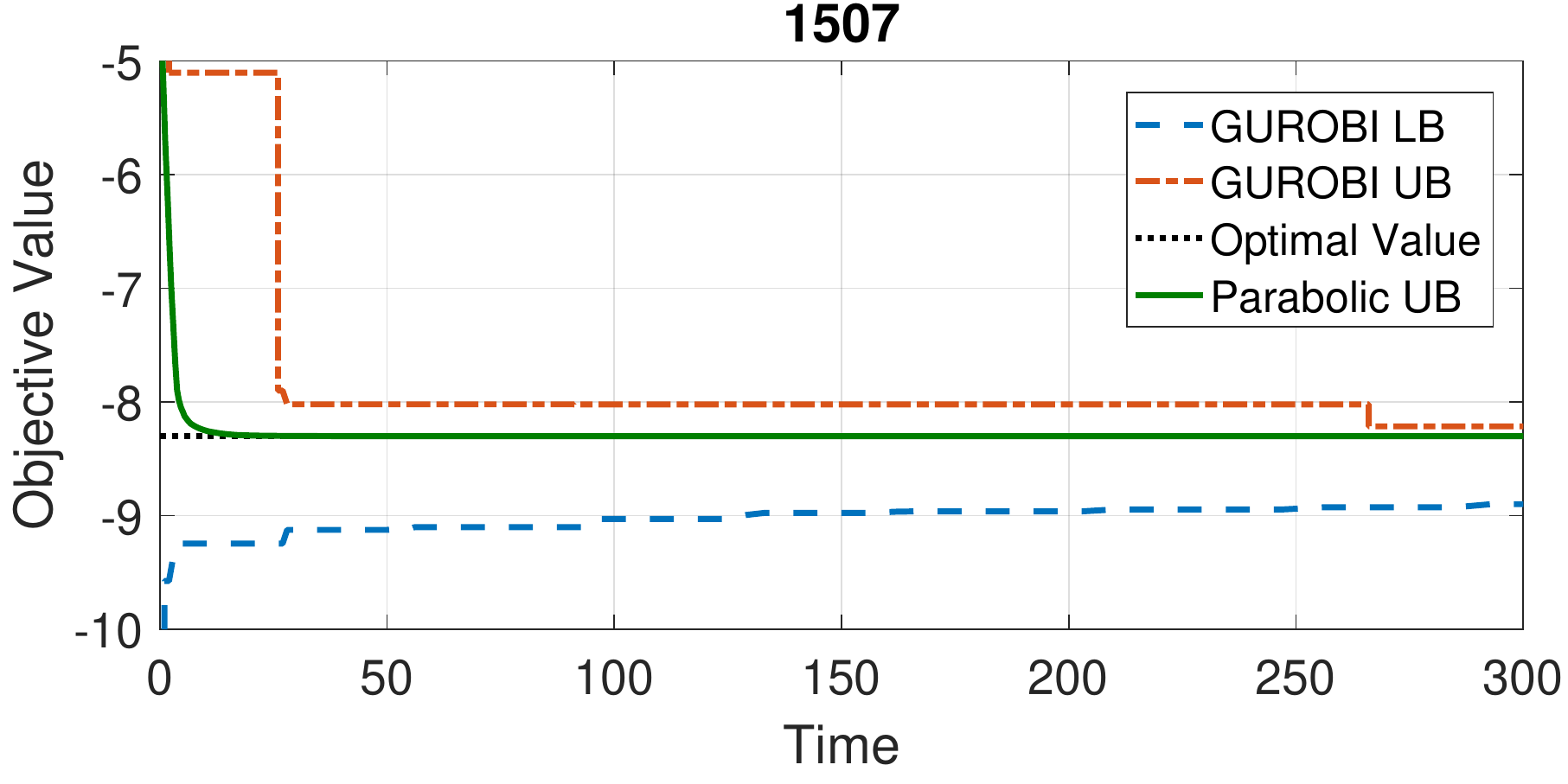}
	\includegraphics[height=0.24\textwidth]{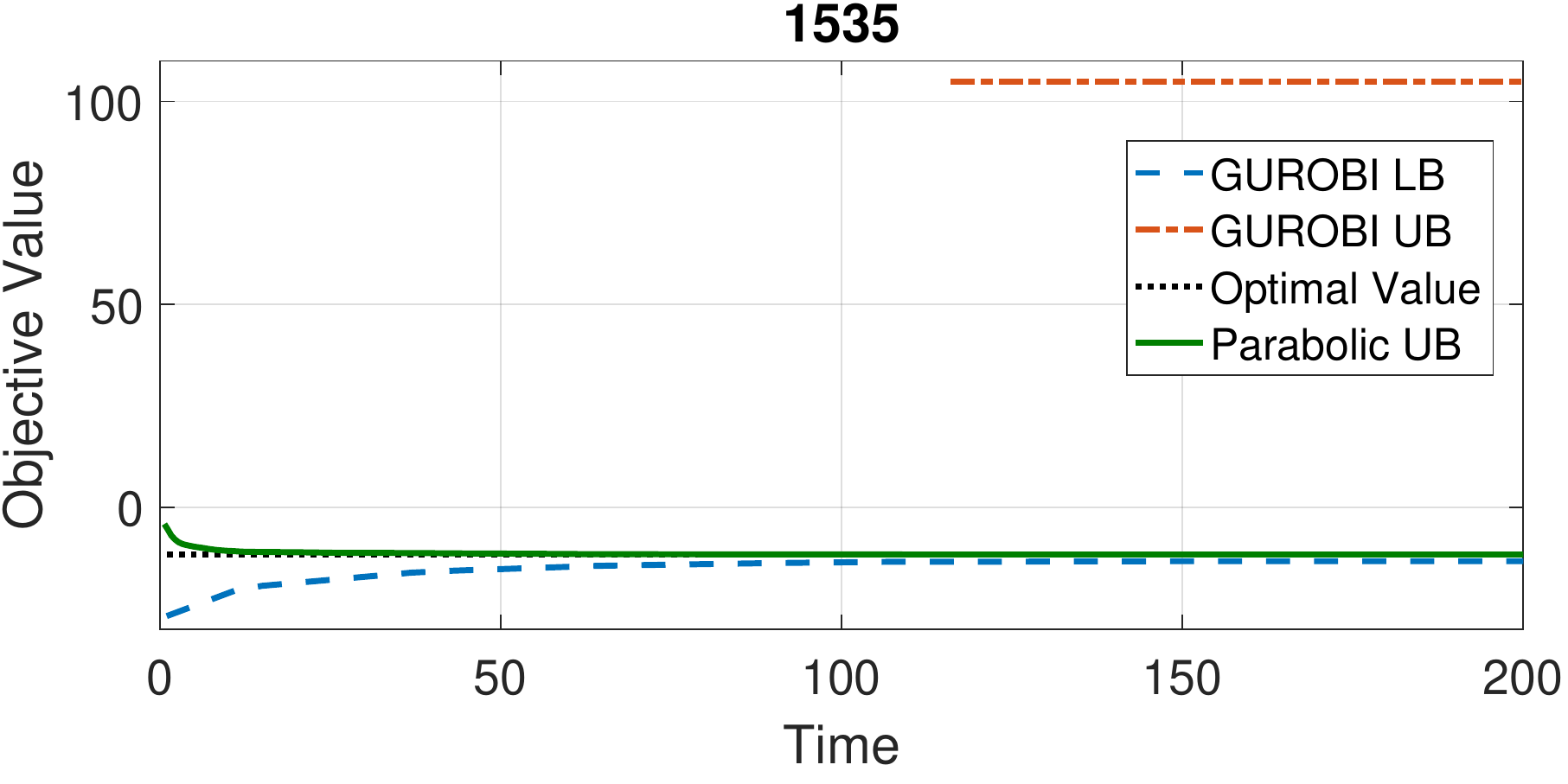}\\
	\caption{\small
		Bounds offered by GUROBI 9.0 and the penalized parabolic relaxation for QPLIB cases.
	}
	\label{fig_good1}
\end{figure*}

\begin{figure*}[h!]
	\centering
	\includegraphics[height=0.24\textwidth]{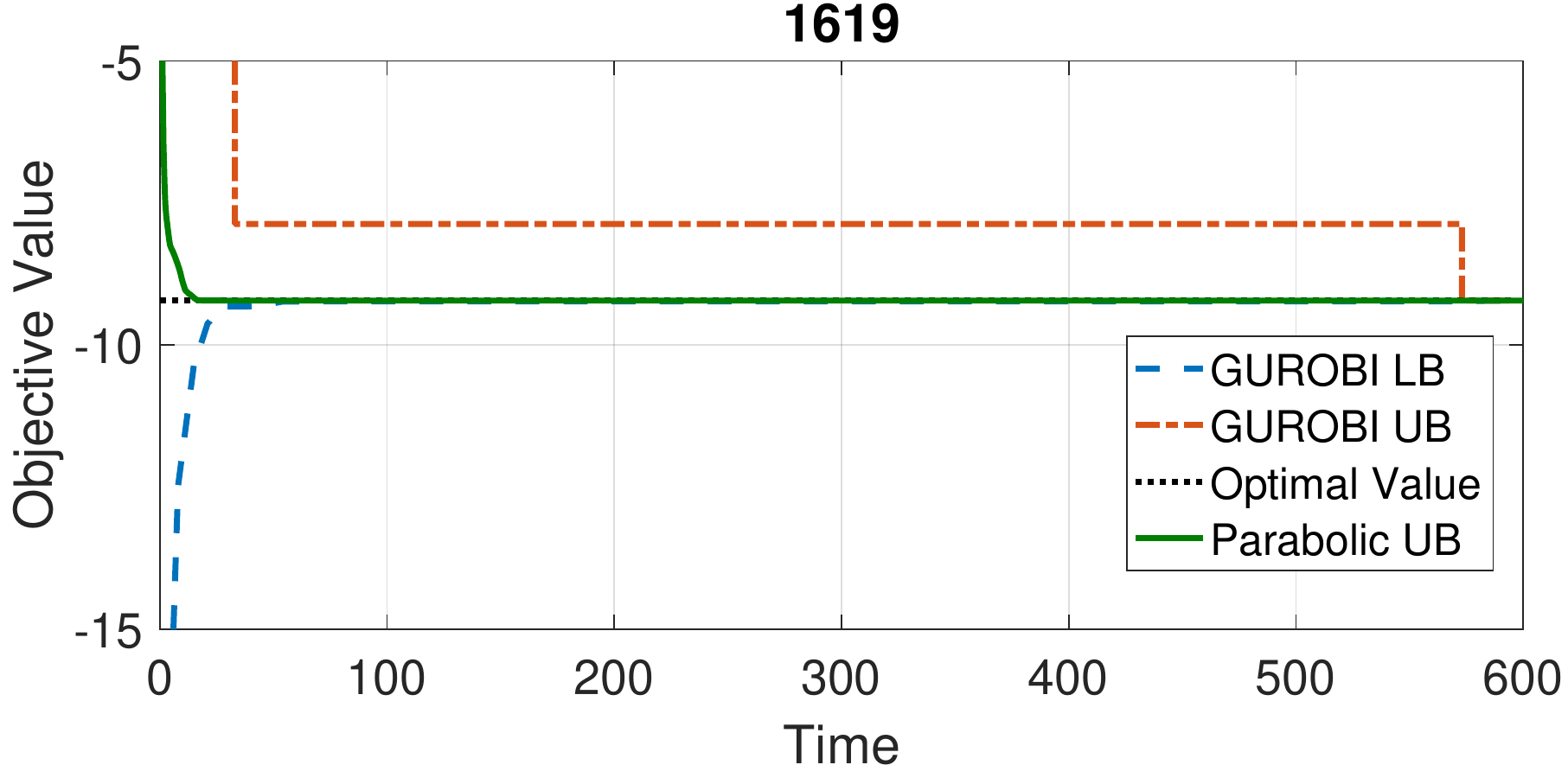}
	\includegraphics[height=0.24\textwidth]{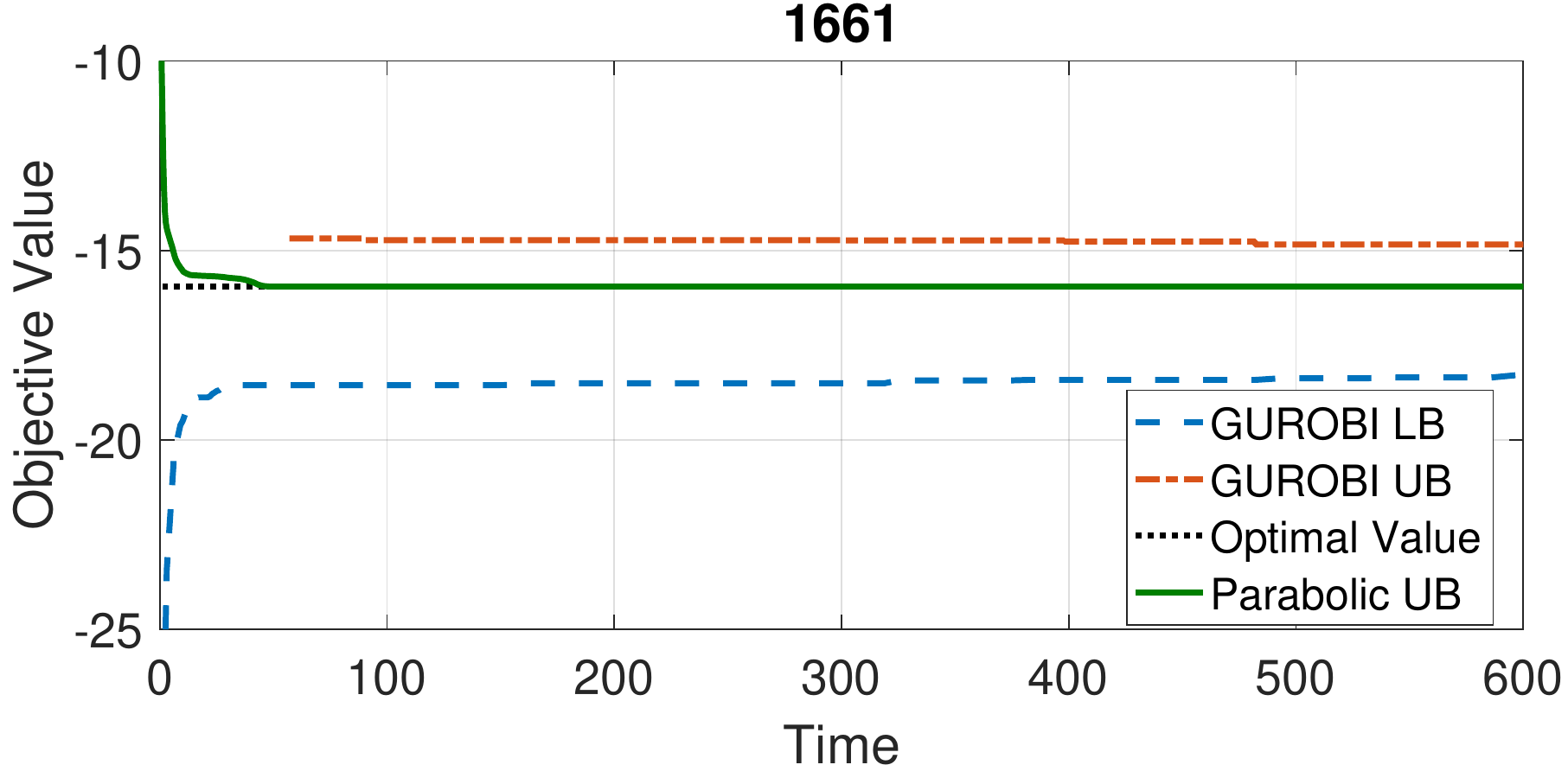}\\
	\vspace{1mm}
	\includegraphics[height=0.24\textwidth]{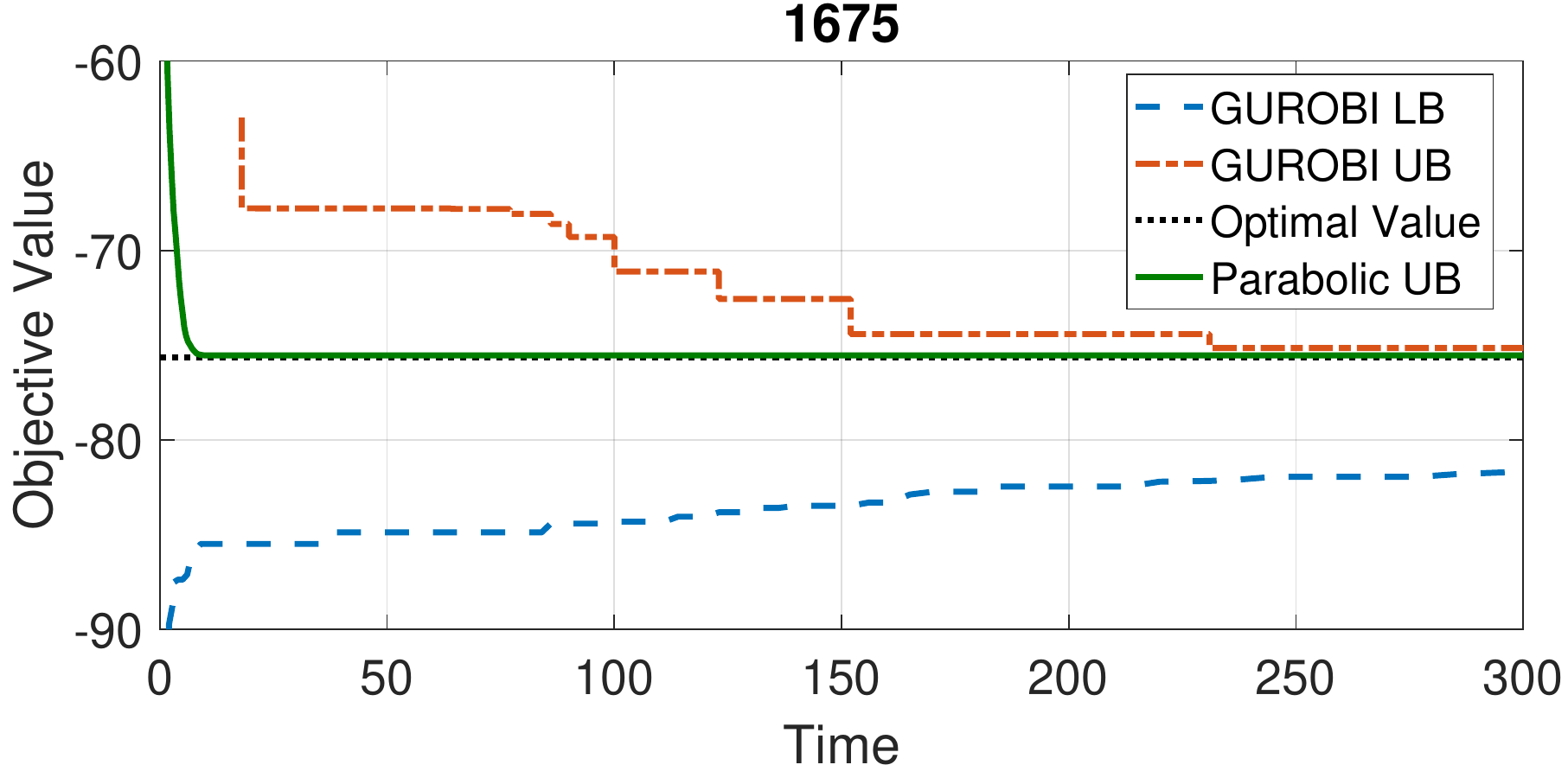}
	\includegraphics[height=0.24\textwidth]{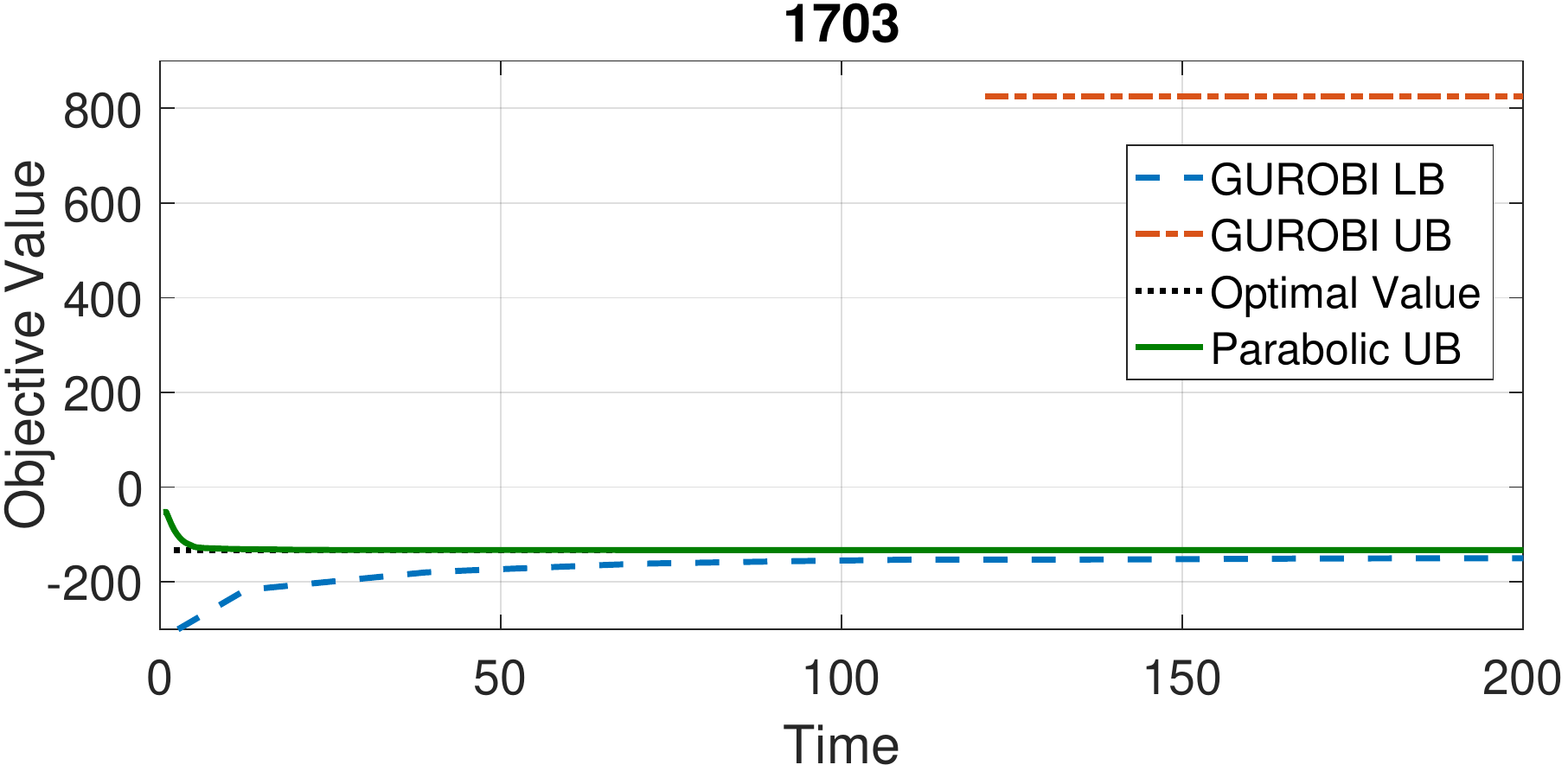}\\
	\vspace{1mm}
	\includegraphics[height=0.24\textwidth]{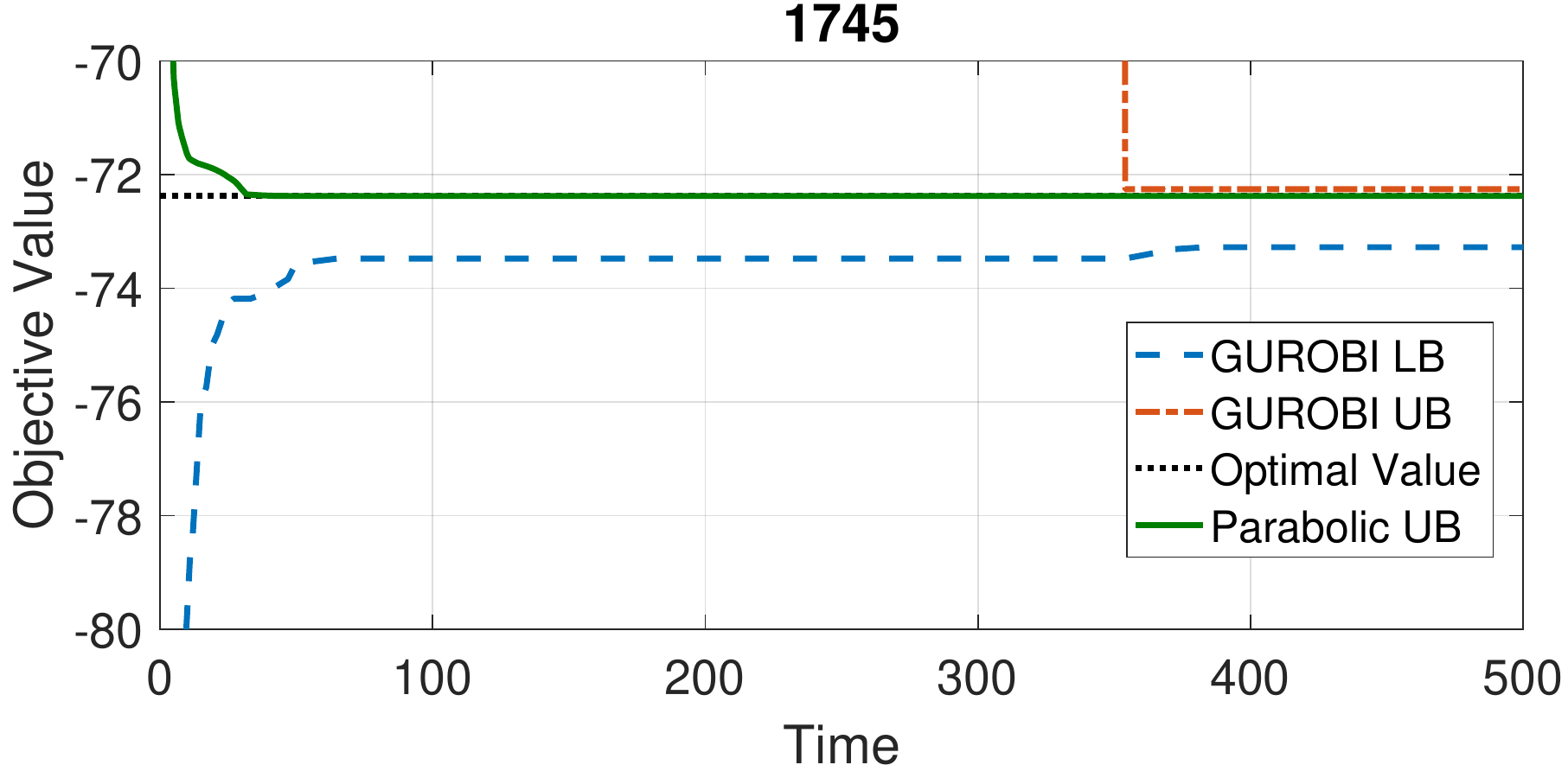}
	\includegraphics[height=0.24\textwidth]{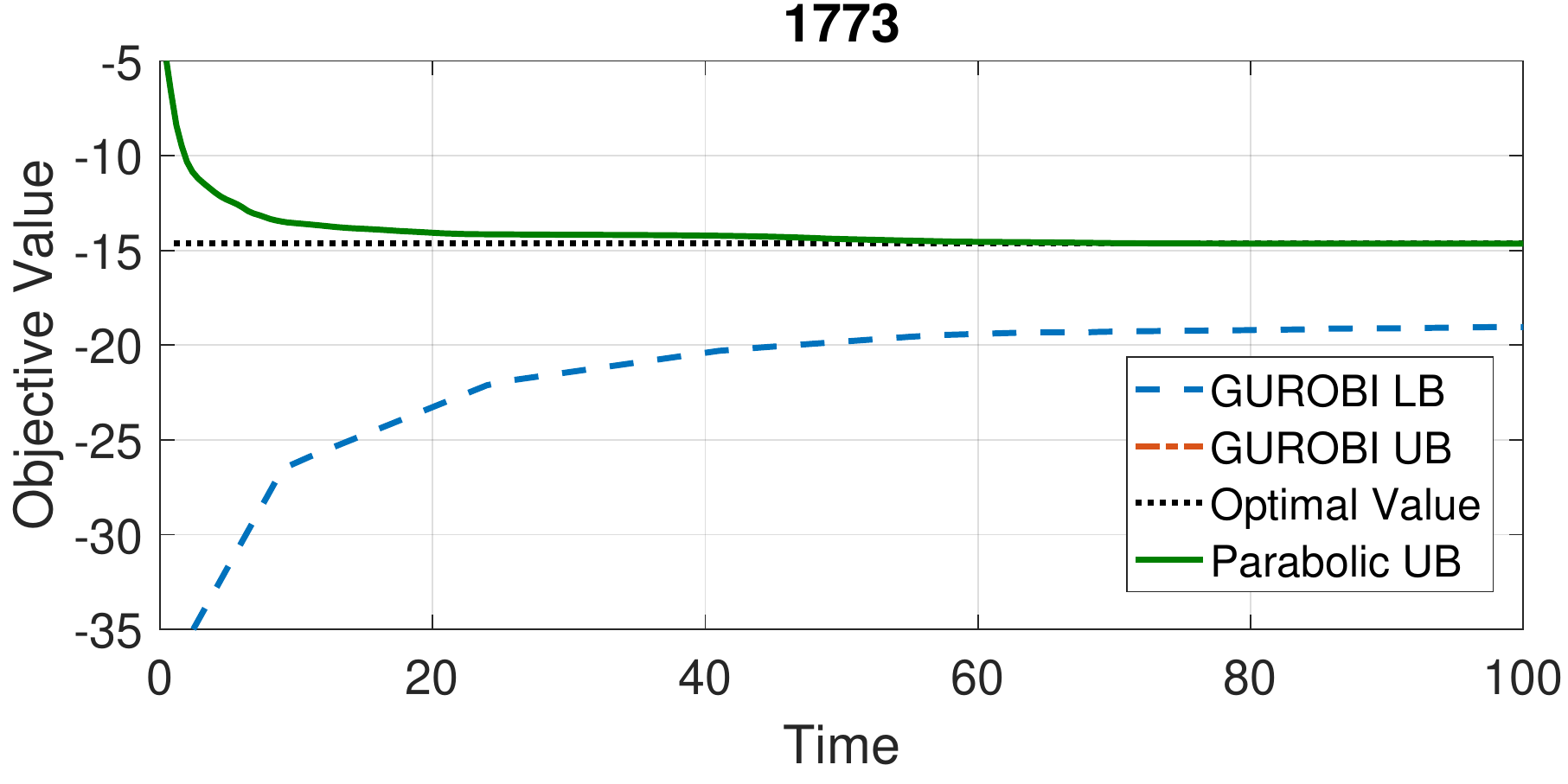}\\
	\vspace{1mm}
	\includegraphics[height=0.24\textwidth]{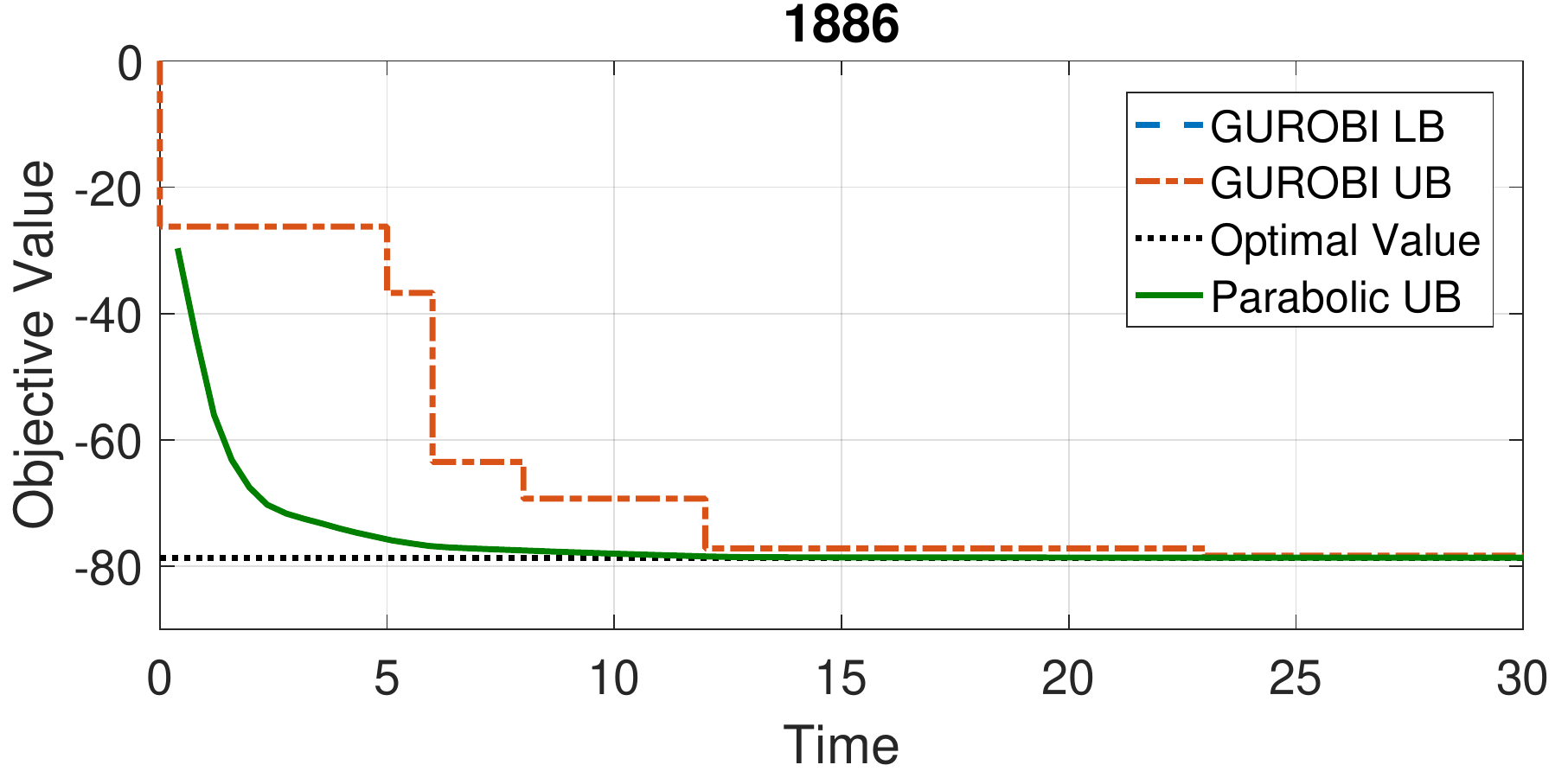}
	\includegraphics[height=0.24\textwidth]{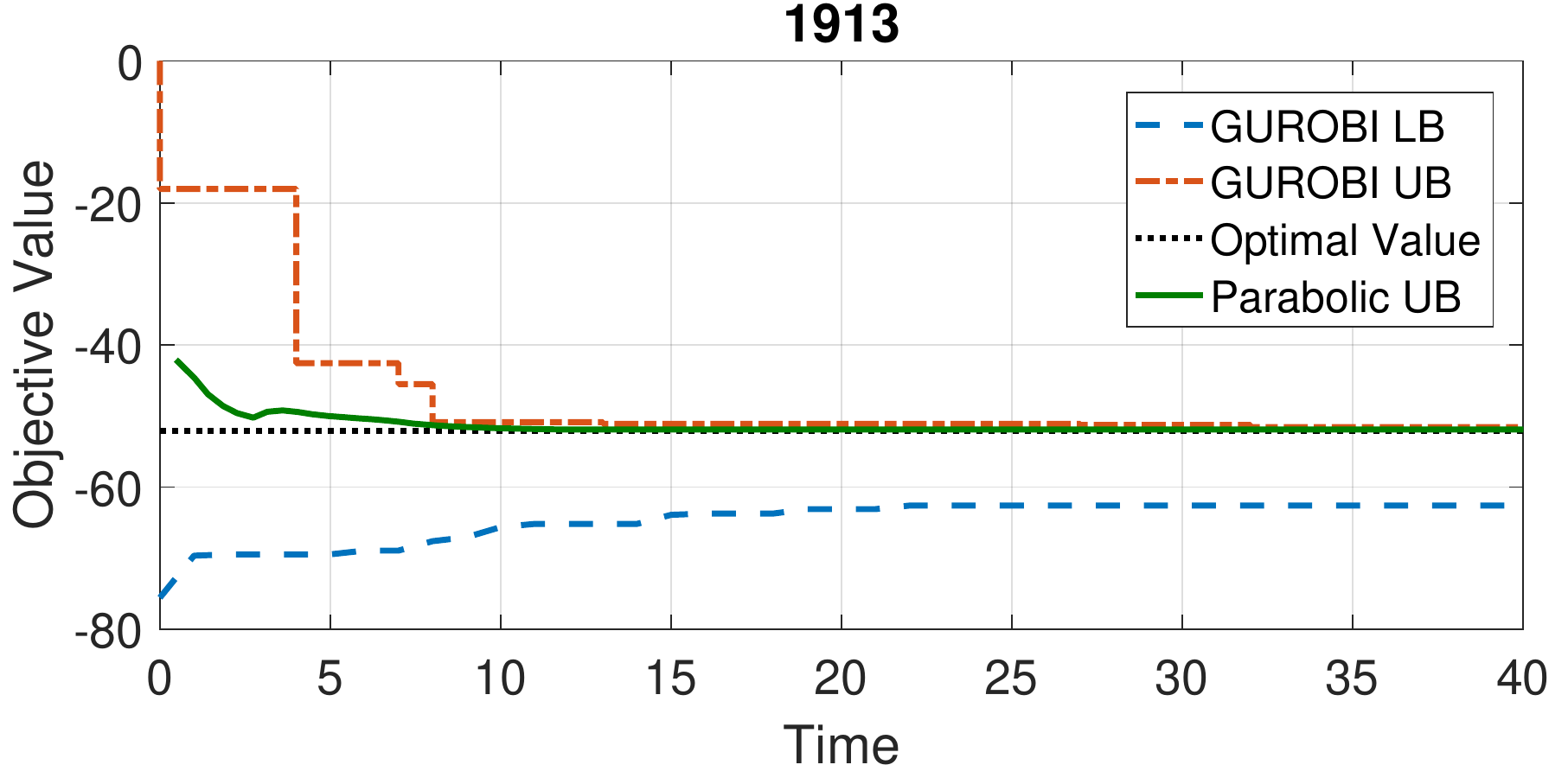}\\
	\vspace{1mm}
	\includegraphics[height=0.24\textwidth]{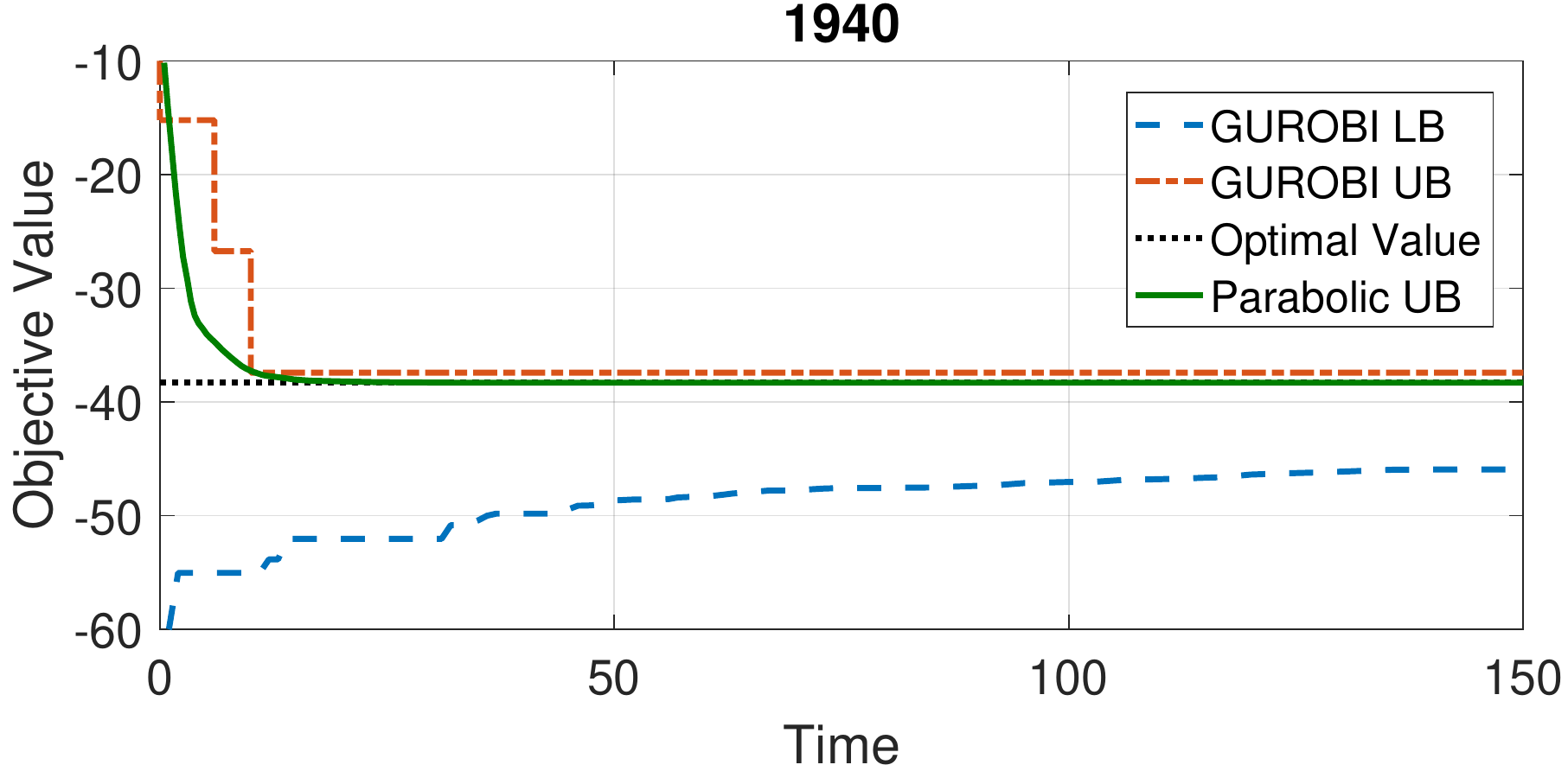}
	\includegraphics[height=0.24\textwidth]{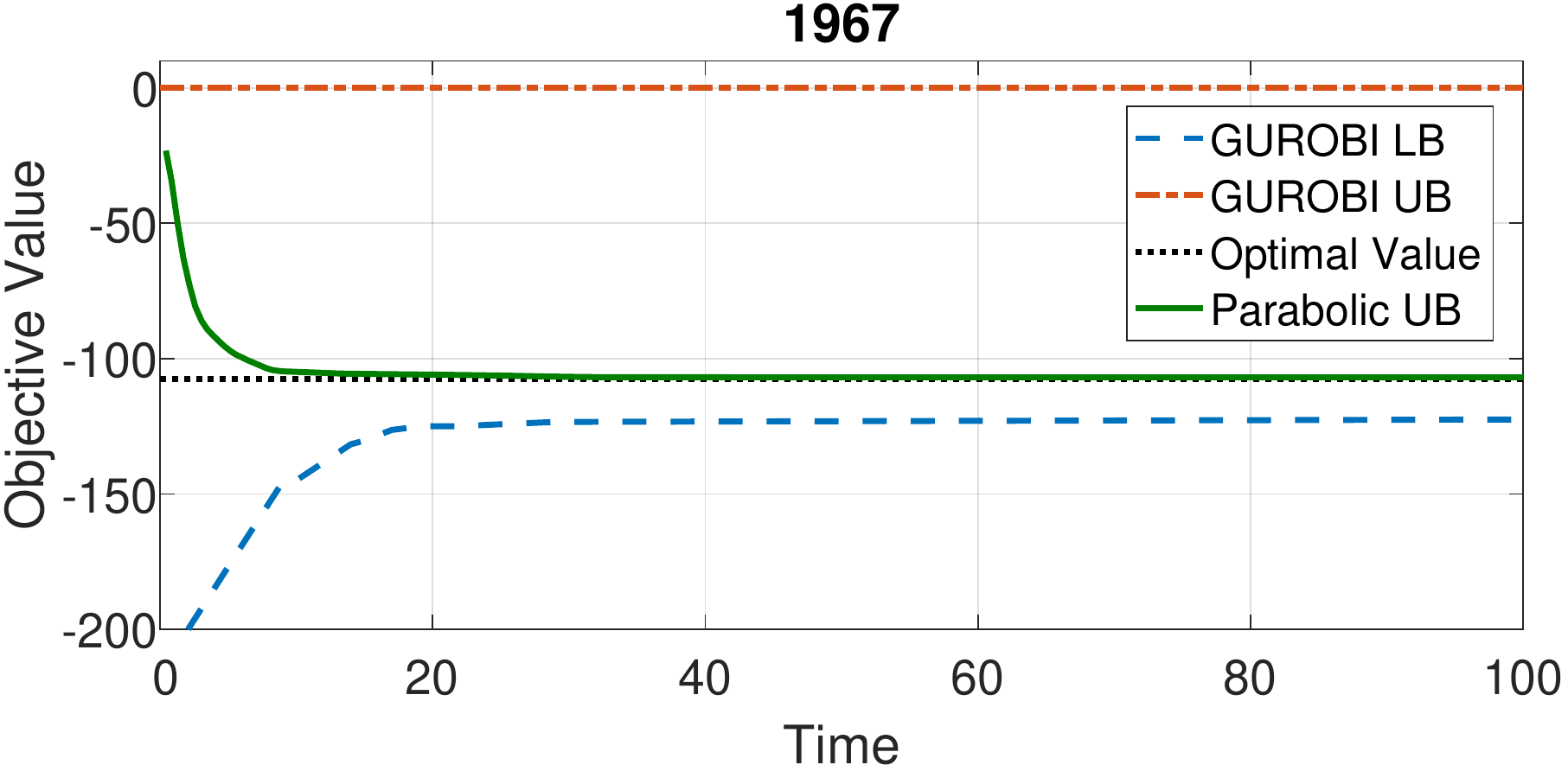}\\
	\caption{\small
		Bounds offered by GUROBI 9.0 and the penalized parabolic relaxation for QPLIB cases.
	}
	\label{fig_good2}
\end{figure*}

\begin{figure*}[t]
	\centering
	\includegraphics[height=0.24\textwidth]{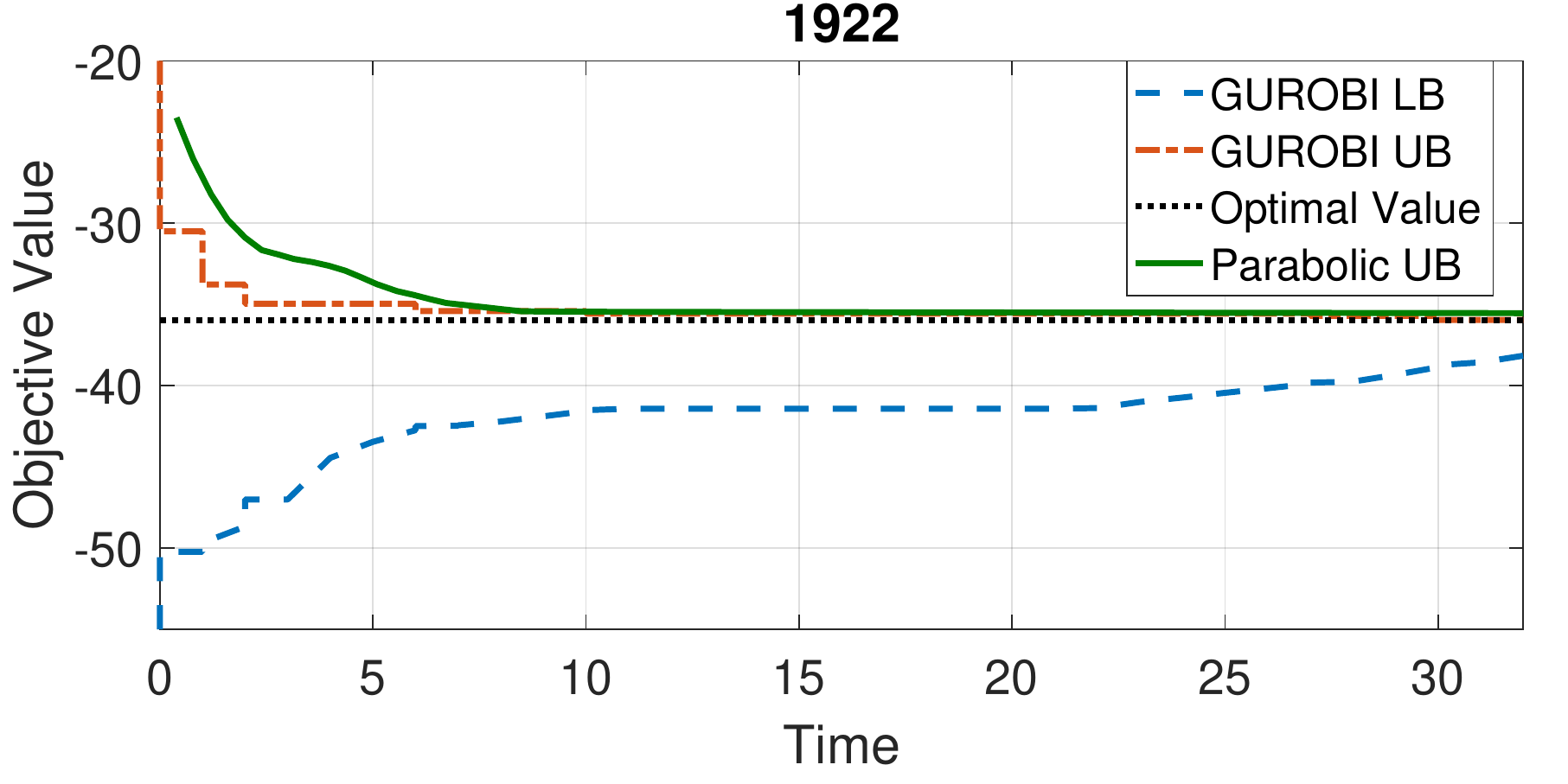}
	\includegraphics[height=0.24\textwidth]{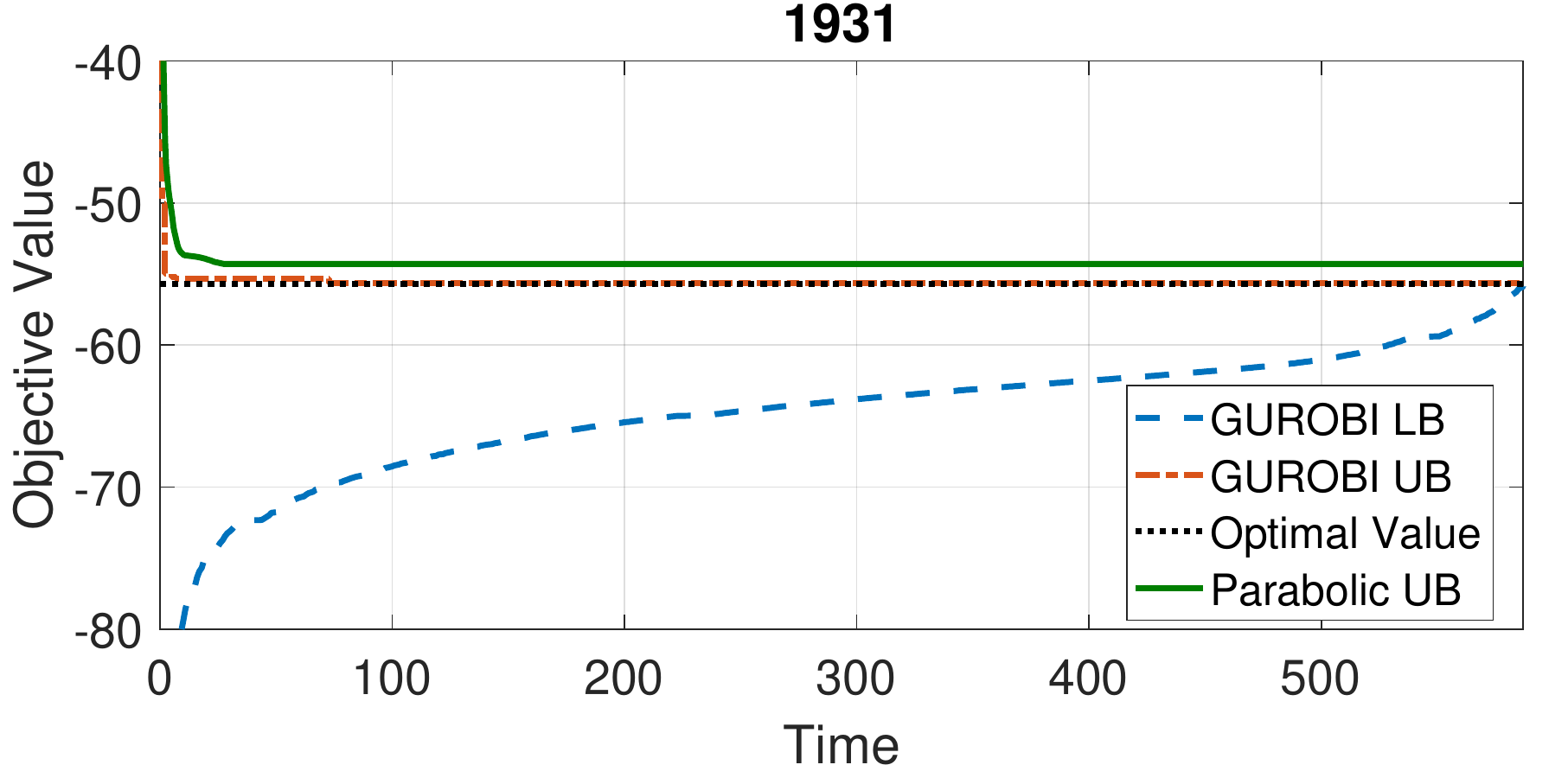}\\
	\vspace{1mm}
	\includegraphics[height=0.24\textwidth]{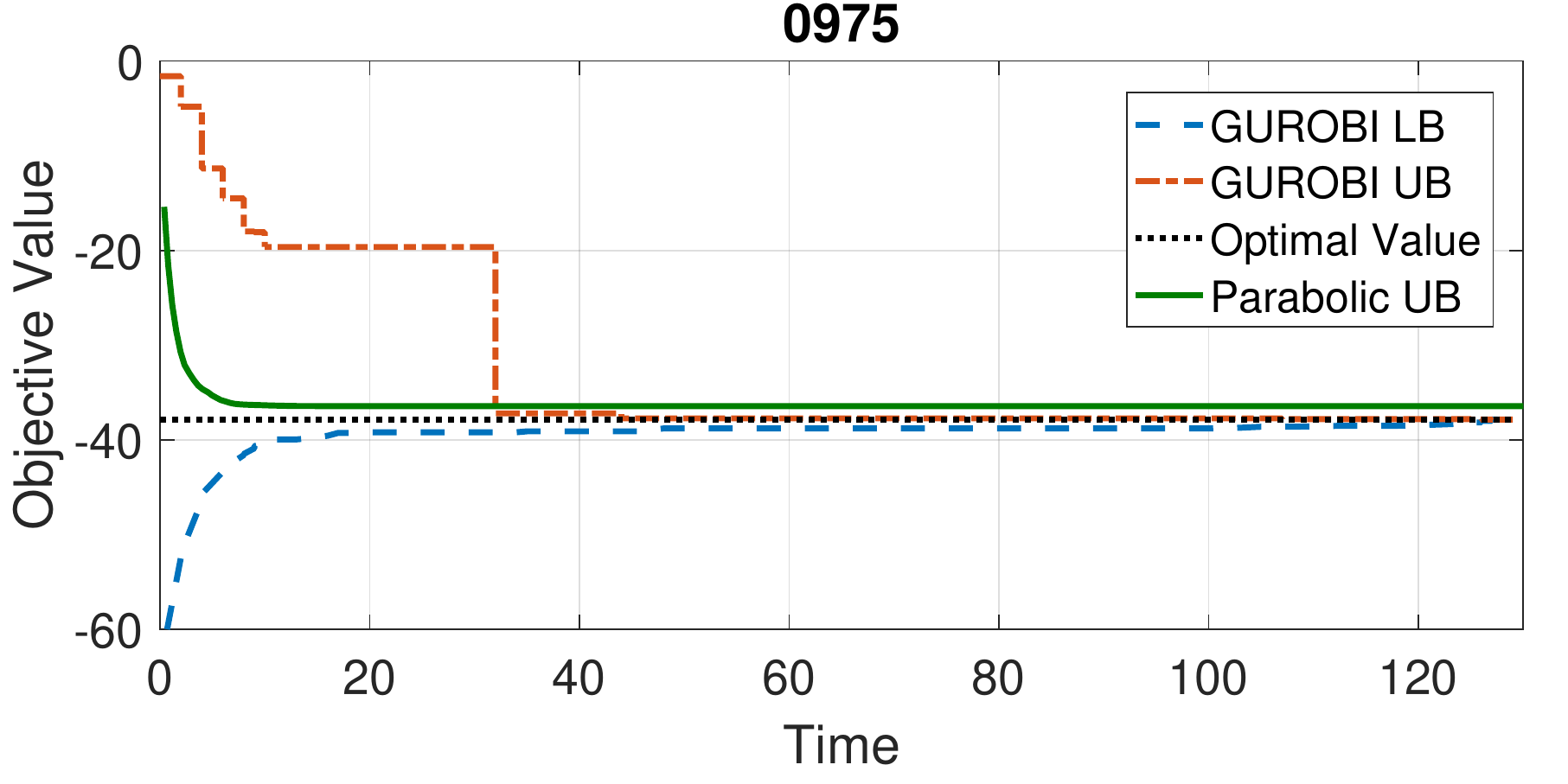}
	\includegraphics[height=0.24\textwidth]{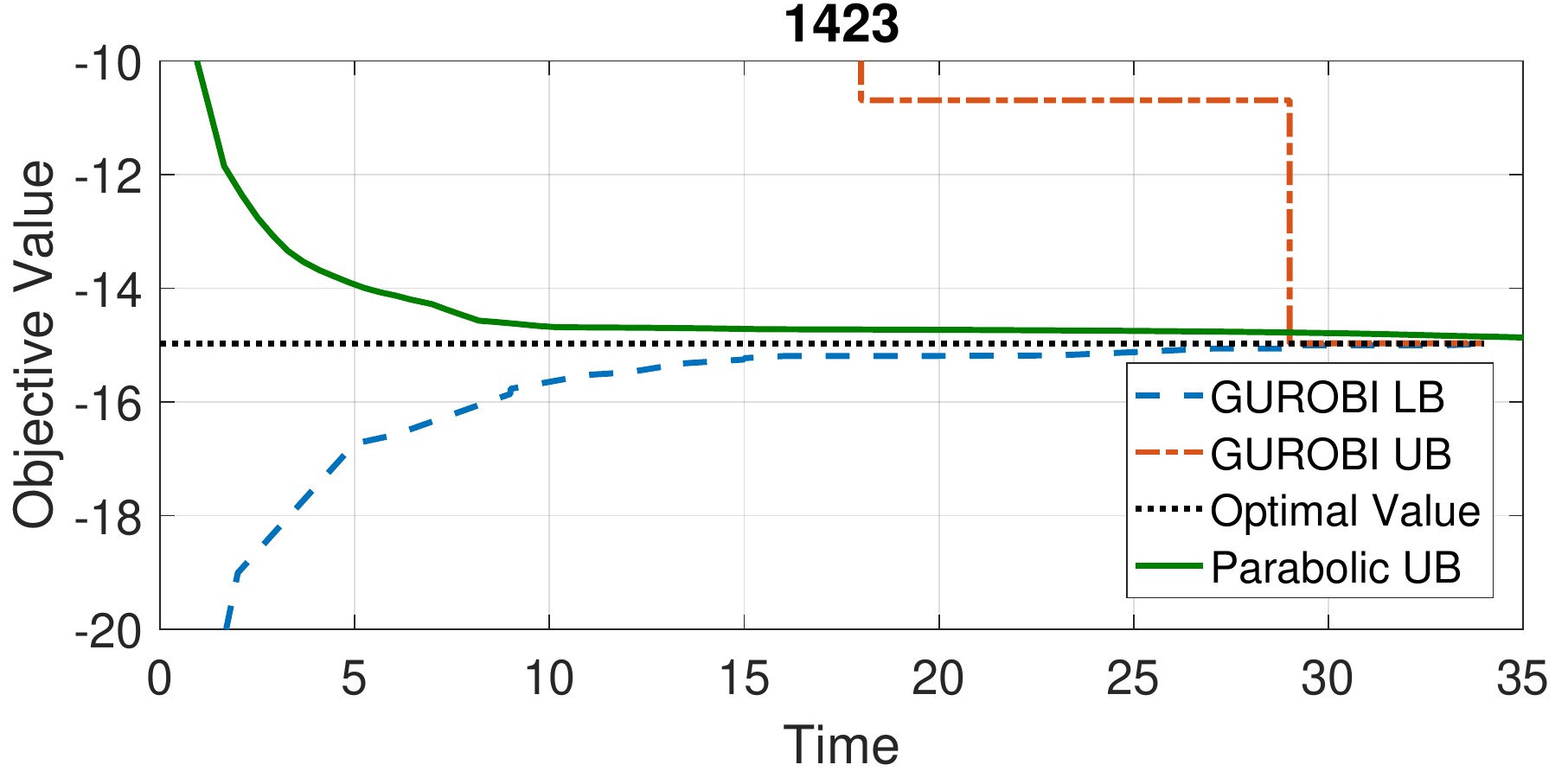}\\
	\caption{\small
		Bounds offered by GUROBI 9.0 and the penalized parabolic relaxation for QPLIB cases.
	}
	\label{fig_bad}
\end{figure*}

\pagebreak

\subsection{Identification of Linear Dynamical Systems}\label{sec_sysid}

Following \cite{MEC} and \cite{fattahi2018data}, this case study is concerned with the problem of identifying the parameters of linear dynamical systems given limited observations of state vectors. Optimization is an important tool for problems involving dynamical systems such as the identification of transfer functions and control synthesis \cite{1_rotkowitz2005characterization, 2_wang2019system,3_wang2018separable,4_kheirandishfard2018convex,5_fattahi2018transformation}.
One of these computationally-hard problems is data-driven system identification (without intrusive means) which has been widely studied in the literature of control \cite{6_sarkar2018fast,7_pereira2010learning}. 
In this case study, we cast system identification as a non-convex QCQP and evaluate ability of the proposed parabolic relaxation in solving very large-scale instances of this problem.
We consider identification problems with over 10,000 variables for which:
\begin{itemize}
\item SDP relaxation is intractable.
\item Both variants of SOCP relaxation as well as $\Scal^{(2+\msf)}_{\nsf,\msf}$ fail to produce feasible points due to the absence of comprehensive boundary property (CBP). \footnote{$\Scal^{(2+\msf)}_{\nsf,\msf}$ and CBP are defined in \cite{parabolic_part1}.}
\item Parabolic relaxation is consistently successful in recovering the ground truth feasible point.
\end{itemize}

To formulate the problem, consider a discrete-time linear system described by the following dynamical equations:
\begin{align}
{\xsf[\tsf+1]}&={\Asf}{\xsf[\tsf]}+{\Bsf}\usf[\tsf],
&& \tsf\in\Tcal,\label{sys}
\end{align}
throughout a time horizon $\Tcal=\{1,2,\ldots,|\Tcal|\}$, where:
\begin{itemize}
\item $\{\xsf[\tsf]\in\mathbb{R}^{\nsf}\}_{\tsf\in\Tcal}$ are the state vectors that are known at times $\tsf\in\Kcal$, where $\Kcal\subseteq\Tcal$ and unknown at times $\tsf\in\Ucal\triangleq\Tcal\setminus\Kcal$. 
\item $\{\usf[\tsf]\in\mathbb{R}^{\msf}\}_{\tsf\in\Tcal}$ are the known control command vectors that are enforced externally to regulate the behavior of system.
\item $\Asf\in\mathbb{R}^{{\nsf}\times {\nsf}}$ and 
$\Bsf\in\mathbb{R}^{{\nsf}\times {\msf}}$ are fixed unknown matrices.
\end{itemize}
Our goal is to determine the pair of ground truth matrices $(\Asf^{\mathrm{true}},\Bsf^{\mathrm{true}})$, given a sample trajectory of the control commands $\{\usf[\tsf]\in\mathbb{R}^{\msf}\}_{\tsf\in\Tcal}$ and the incomplete state vectors $\{\xsf[\tsf]\in\mathbb{R}^{\nsf}\}_{\tsf\in\Kcal}$. This problem amounts to the following optimization with non-convex quadratic constraints:
\begin{subequations}
\begin{align}
	&\text{find}
	&&\hspace{-0.25cm} 
	\{\xsf[\tsf],\zsf[\tsf]\in\Rbb^{\nsf}\}_{\tsf\in\Ucal},\quad
	\Asf\in\Rbb^{{\nsf}\times {\nsf}},\quad
	\Bsf\in\Rbb^{{\nsf}\times {\msf}}
	\label{lav_obj}\\
	& \text{subject to}
	&&\hspace{-0.25cm} 
	\xsf[\tsf+1]=\Asf\xsf[\tsf]+\Bsf\usf[\tsf],
	&&&&\tsf\in\Kcal,\label{lav_cons}\\
	& 
	&&\hspace{-0.25cm} 
	\xsf[\tsf+1]=\hspace{-2.0mm}\underbrace{\Asf\xsf[\tsf]}_{\text{non-convex}}\hspace{-1.75mm}+\hspace{0.75mm} \Bsf\usf[\tsf],
	&&&&\tsf\in\Ucal.\label{lav_cons_2}
\end{align}
\end{subequations}
Observe that the constraint \eqref{lav_cons} is linear because $\xsf[\tsf]$ is known for every $\tsf\in\Kcal$, whereas the constraint \eqref{lav_cons_2} is non-convex due to the bilinear term $\Asf\xsf[\tsf]$ because when  $\tsf\in\Ucal$, both $\xsf[\tsf]$ and $\Asf$ are unknown. Problem \eqref{lav_obj} -- \eqref{lav_cons_2} can be cast in the lifted form with respect to the matrix variable: 
\begin{align}
\Ybf\triangleq
\begin{bmatrix}
	\Asf\T\;\;
	\big[\xsf[\tsf]\big]_{\tsf\in\Ucal}
\end{bmatrix}\T\in\Rbb^{({\nsf}+|\Ucal|)\times {\nsf}},
\end{align}
which is composed of the matrix $\Asf$ and all of the unknown state vectors:
\begin{subequations}
\begin{align}
	&\text{find}
	&&
	\{\xsf[\tsf]\in\Rbb^{\nsf}\}_{\tsf\in\Ucal},\quad
	\Asf\in\Rbb^{{\nsf}\times {\nsf}},\quad
	\Bsf\in\Rbb^{{\nsf}\times {\msf}},\quad
	\bar{\Asf}\in\Sbb_{\nsf},\quad
	\bar{\Xsf}\in\Sbb_{|\Ucal|}\\
	& \text{subject to}
	&&
	\xsf[\tsf+1]=\Asf\xsf[\tsf]+\Bsf\usf[\tsf],
	&&&&\hspace{-2.25cm} \tsf\in\Kcal,\\
	& 
	&&
	\!\begin{bmatrix}
		\bar{\Asf} & 
		\big[\xsf[\tsf+1]-\Bsf\usf[\tsf]\big]_{\tsf\in\Ucal}\\
		&\vspace{-3mm}\\
		\ast & \bar{\Xsf}
	\end{bmatrix}\!=\!
	\begin{bmatrix}
		\Asf\\
		\vspace{-3mm}\\
		\big[\xsf[\tsf]\big]_{\tsf\in\Ucal}\T
	\end{bmatrix}
	\begin{bmatrix}
		\Asf\T\;\;
		\big[\xsf[\tsf]\big]_{\tsf\in\Ucal}
	\end{bmatrix},
	&&&&\hspace{-1.25cm} \tsf\in\Ucal,\label{sysid_sdp}
\end{align}
\end{subequations}
where $\bar{\Asf}\in\Sbb_{\nsf}$ and $\bar{\Xsf}\in\Sbb_{|\Ucal|}$ account for $\Asf\Asf\T$ and $\big[\xsf[\tsf]\big]_{\tsf\in\Ucal}\T\big[\xsf[\tsf]\big]_{\tsf\in\Ucal}$, respectively. Transformation of ``$=$'' to ``$\succeq$'' results in an SDP relaxation of this problem. However, this relaxation can be intractable because of the dense bipartite clique $\big[\xsf[\tsf+1]-\Bsf\usf[\tsf]\big]_{\tsf\in\Ucal}$. Alternatively, parabolic relaxation allows us to omit all of the off-diagonal elements of $\bar{\Asf}$ and $\bar{\Xsf}$ from the lifted formulation. Hence, in order to formulate parabolic relaxation of problem \eqref{lav_obj} -- \eqref{lav_cons} we only need to introduce a modest number of auxiliary variables:
\begin{itemize}
\item Define $\bar{\boldsymbol{\asf}}=\mathrm{diag}\{\bar{\Asf}\}\in\Rbb^{\nsf}$ as the variable whose $k$-th element represents $\|\Asf\T\ebf_k\|_2^2$, where $\{\ebf_k\}^{\nsf}_{k=1}$ is the standard base in $\Rbb^{\nsf}$.
\item Define $\bar{\xsf}=\mathrm{diag}\{\bar{\Xsf}\}\in\Rbb^{|\Ucal|}$ as the variable whose elements represent $\|\xsf[\tsf]\|_2^2$ for all $\tsf\in\Ucal$,
\end{itemize}
Then, problem \eqref{lav_obj} -- \eqref{lav_cons} can be relaxed as:
\begin{subequations}
\begin{align}
	&\text{find}
	&&
	\{\xsf[\tsf]\in\Rbb^{\nsf}\}_{\tsf\in\Ucal},\quad
	\Asf\in\Rbb^{{\nsf}\times {\nsf}},\quad
	\Bsf\in\Rbb^{{\nsf}\times {\msf}},\quad
	\bar{\asf}\in\Rbb^{\nsf},\quad
	\bar{\xsf}\in\Rbb^{|\Ucal|}\\
	& \text{subject to}
	&&
	\xsf[\tsf+1]=\Asf\xsf[\tsf]+\Bsf\usf[\tsf],
	&&&&\hspace{-4.15cm} \tsf\in\Kcal,\\
	& &&\hspace{-1.7cm} 
	\bar{\mathsf{a}}_k+\bar{\mathsf{x}}_{\tsf}
	+2\,\ebf\T_k\big(\xsf[\tsf+1]-\Bsf\usf[\tsf]\big)
	\geq\|\Asf\T\ebf_k+\xsf[\tsf]\|_2^2,
	&&&&\hspace{-4.15cm} \tsf\in\Ucal,\; k\in\{1,\ldots,{\nsf}\},\\
	& &&\hspace{-1.7cm} 
	\bar{\mathsf{a}}_k+\bar{\mathsf{x}}_{\tsf}
	-2\,\ebf\T_k\big(\xsf[\tsf+1]-\Bsf\usf[\tsf]\big)
	\geq\|\Asf\T\ebf_k-\xsf[\tsf]\|_2^2,
	&&&&\hspace{-4.15cm} \tsf\in\Ucal,\; k\in\{1,\ldots,{\nsf}\},\\
	& &&\hspace{-1.7cm}
	\bar{\mathsf{x}}_t \,\geq\|\xsf[\tsf]\|_2^2,
	&&&&\hspace{-4.15cm} \tsf\in\Ucal,\\
	& &&\hspace{-1.7cm}
	\bar{\mathsf{a}}_k \geq\|\Asf\T\ebf_k\|_2^2,
	&&&&\hspace{-4.15cm} k\in\{1,\ldots,{\nsf}\}.
\end{align}
\end{subequations}

In this case study, we consider system identification problems with ${\nsf}=16$, ${\msf}=14$, $|\Tcal|=801$ and $\Kcal=\{1,5,9,\ldots,801\}$, i.e., one in every four state vectors is assumed known. The resulting QCQP variable $\boldsymbol{x}$ is $10,096$-dimensional. The ground truth matrices are chosen as follows:
\begin{itemize}
\item Every element of $\Asf-\Ibf_{\! \nsf}$ is uniformly chosen from the interval $[-0.25,+0.25]$
\item Elements of $\Bsf\in\mathbb{R}^{16\times 14}$ have zero-mean standard normal distribution.
\item Elements of $\xsf[1]$ are uniformly chosen from the interval $[0.5,1.5]$.
\item For every $\tsf\in\Tcal$, we have $\usf[\tsf]=\Fsf\xsf[\tsf]+\wsf[\tsf]$, where the elements of $\wsf[\tsf]$ have independent Gaussian distribution with zero mean and standard deviation $0.1$. Additionally, $\Fsf=-(\Ibf_{\! \msf}+\Bsf^{\top}\Psf)^{-1}\Bsf^{\top}\Psf\Asf$ is an optimal LQR controller with $\Psf$ representing the unique positive-definite solution to the Riccati equation:
\begin{align}
	\Asf^{\!\top}\Psf\Asf + \Ibf_{\nsf} - \Psf = \Asf^{\!\top}\Psf\Bsf (\Ibf_{\nsf} + \Bsf^{\!\top}\Psf\Bsf)^{-1}  \Bsf^{\!\top}\Psf\Asf.
\end{align}
\end{itemize}
We generate 15 random systems and solve the resulting problem, using the sequential Algorithm \ref{alg:1} with the initial point $\check{\Asf}=\Ibf_{\nsf}$ and $\check{\xsf}[\tsf]=\boldsymbol{0}_{\nsf}$ for every $t\in\Ucal$.
Figure \ref{fig_SYSID} illustrates the convergence of the error function:
\begin{align}
\frac{1}{\nsf}\|\Asf-\Asf^{\mathrm{true}}\|_{\Frm}+ \frac{1}{\sqrt{\nsf\msf}}\|\Bsf-\Bsf^{\mathrm{true}}\|_{\Frm},
\end{align}
for the 15 instances. Each round of the penalized parabolic is solved in less than 15 seconds. Due to the bipartite ${\nsf}\times|\Ucal|$ clique in \eqref{sysid_sdp}, SDP relaxation is intractable even when sparsity is exploited.
Moreover, due to the absence of CBP, neither of SOCP relaxations variants in \cite[Section 5]{parabolic_part1} converge to the true system matrices.


\begin{figure*}[h!]
\centering
\includegraphics[height=0.35\textwidth]{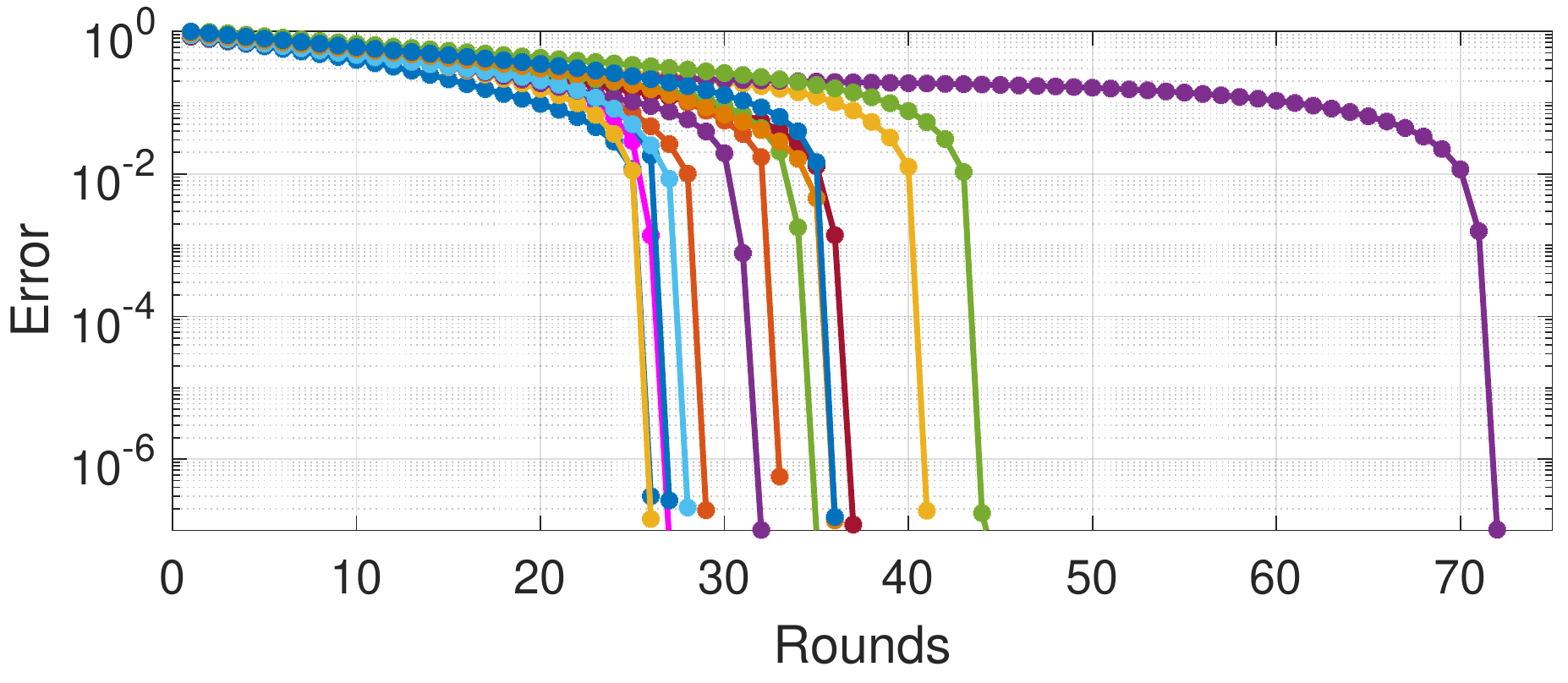}
\caption{\small
	The performance of the sequential Algorithm \ref{alg:1} for system identification.
}
\label{fig_SYSID}
\end{figure*}

\section{Conclusions}
In the first this work \cite{parabolic_part1}, we proposed a computationally efficient convex relaxation, named {\it parabolic relaxation}, to find near optimal solutions for quadratically-constrained quadratic programming. 
In cases where the relaxation is not exact, a penalized relaxation is developed to steer the relaxation towards a feasible solution. To improve the quality of the recovered solution, we propose a sequential scheme that starts from an arbitrary point (feasible or infeasible) and solves a sequence of penalized problems to recover a feasible near optimal solution. We prove the convergence of the sequential procedure to a KTT point. The experimental results show the effectiveness of our proposed approach on benchmark QCQP cases as well as the problem of identifying linear dynamical systems.

\section{Acknowledgments}
We are grateful to GAMS Development Corporation for providing us with unrestricted access to a full set of solvers throughout the project. 

\bibliographystyle{spbasic}
\bibliography{references}

\end{document}